\documentclass[10pt, letterpaper]{amsart}

\usepackage{amsthm,amsmath,amsfonts,amssymb,amscd,mathrsfs}
\usepackage{tikz, graphicx,color,geometry}
\usepackage{todonotes}
\usepackage[all]{xy}
\usepackage{hyperref}

\newtheorem{thm}{Theorem}[section]

\newtheorem{lem}[thm]{Lemma}
\newtheorem{cor}[thm]{Corollary}

\newtheorem{innercustomthm}{Theorem}
\newenvironment{customthm}[1]
{\renewcommand\theinnercustomthm{#1}\innercustomthm}
{\endinnercustomthm}

\theoremstyle{remark}
\newtheorem{rem}[thm]{Remark}
\newtheorem{notat}{Notation}

\providecommand*{\propertyautorefname}{Property}
\newtheorem{property}{\propertyautorefname}

\newcommand{\RR}{\mathbb{R}}
\newcommand{\CC}{\mathbb C}

\newcommand{\ZZ}{\mathbb{Z}}

\newcommand{\PP}{\mathbb{P}}

\newcommand{\cX}{\mathcal{X}}
\newcommand{\cY}{\mathcal{Y}}
\newcommand{\cM}{\mathcal{M}}
\newcommand{\cO}{\mathcal{O}}

\newcommand{\cC}{\mathcal{C}}

\newcommand{\sA}{\mathscr{A}}
\newcommand{\sB}{\mathscr{B}}

\newcommand{\sQ}{\mathscr{Q}}

\newcommand{\fjrw}[2]{ \left\lceil #1 \:; #2 \right\rfloor }
\newcommand{\jw}{\mathfrak{j}}
\newcommand{\set}[1]{\{#1\}}
\newcommand{\tw}[1]{{#1}_{tw}}
\newcommand{\s}[1]{\Sigma #1}

\newcommand{\mb}[1]{\mathbf{#1}}

\DeclareMathOperator{\Hess}{Hess}
\DeclareMathOperator{\SL}{SL}
\DeclareMathOperator{\GL}{GL}
\DeclareMathOperator{\lcm}{lcm}
\DeclareMathOperator{\Fix}{Fix}
\DeclareMathOperator{\age}{age}

\definecolor{pistachio}{rgb}{0.58, 0.77, 0.45}
\definecolor{red(munsell)}{rgb}{0.95, 0.0, 0.24}
\definecolor{eggshell}{rgb}{0.94, 0.92, 0.84}

\begin{document}
\date{\today}

\title{Borcea--Voisin mirror symmetry for Landau--Ginzburg models}

\author{Amanda Francis}
\address{Mathematical Reviews, American Mathematical Society, Ann Arbor, MI 48103, USA}
\email{aefr@umich.edu}
\author{Nathan Priddis}
\address{Institut f\"ur Algebraische Geometrie, Leibniz Universit\"at Hannover, 30167 Hannover, Germany}
\curraddr{Mathematics Department, Brigham Young University, Provo, UT 84604, USA}
\email{priddis@math.byu.edu}
\author{Andrew Schaug}
\address{Department of Mathematics, University of Michigan, Ann Arbor, MI 48109, USA}
\email{trygve@umich.edu}

\begin{abstract}
FJRW theory is a formulation of physical Landau--Ginzburg models with a rich algebraic structure, rooted in enumerative geometry. As a consequence of a major physical conjecture, called the Landau--Ginzburg/Calabi--Yau correspondence, several birational morphisms of Calabi-Yau orbifolds should correspond to isomorphisms in FJRW theory. In this paper, we exhibit some of these isomorphisms that are related to Borcea--Voisin mirror symmetry. In particular, we develop a modified version of BHK mirror symmetry for certain LG models. Using these isomorphisms, we prove several interesting consequences in the corresponding geometries. 
\end{abstract}

\maketitle

\section{Introduction}

Mirror symmetry is a phenomenon first described nearly thirty years ago by physicists. One can think of mirror symmetry as an exchange of information determined by the K\"ahler structure of some physical model with information determined by the complex structure of another physical model. 

Mathematically, several different constructions have been described for various models, each in a different context. In this paper, we will focus on two particular constructions of mirror symmetry. The first, which we call Borcea--Voisin (BV) mirror symmetry, relates two Calabi--Yau (CY) threefolds of a particular form, and the second, which has come to be known as Berglund--H\"ubsch--Krawitz (BHK) mirror symmetry, relates two Landau--Ginzburg (LG) models. We extend BHK mirror symmetry to a new formulation of mirror symmetry for LG models based on the ideas of BV mirror symmetry and the Landau--Ginzburg/Calabi--Yau (LG/CY) correspondence. The main theorem of this work establishes an isomorphism of state spaces for our LG mirror models.  


It is generally expected that computations should be less difficult on the LG side of the LG/CY correspondence, thus we expect our new form of mirror symmetry to improve our ability to make calculations for Borcea--Voisin mirror symmetry. This work gives credence to this expectation; we look to the geometry for inspiration, but make most calculations on the LG model. In particular, we will prove several geometric consequences of this new version of mirror symmetry, which are otherwise difficult to obtain, 
including a generalization of mirror symmetry for (geometric) BV models. We end with a result regarding the Frobenius structure of the LG model state spaces. 

The mirror symmetry construction which we call Borcea--Voisin mirror symmetry was first described by Borcea in \cite{Borcea} and Voisin in \cite{Voisin} in the early 1990's. In particular, they constructed a class of Calabi--Yau threefolds by first taking the product of an elliptic curve $E$ with involution $\sigma_E$, and a K3 surface $S$ with involution $\sigma_S$ and constructing the quotient $(E \times S)/\sigma$ by the involution $\sigma = (\sigma_E , \sigma_S)$. 
The construction known as BV mirror symmetry is defined for this particular class of Calabi--Yau threefolds. 

BV mirror symmetry relies on mirror symmetry for K3 surfaces with involution, as described in the work of Nikulin (see \cite{Ni}). That is, to each K3 surface $S$ with involution, there is a family of K3 surfaces that are mirror dual to $S$.

With this in mind, the mirror symmetry posited by Borcea and Voisin is described as follows. Let $S'$ be a K3 surface with involution $\sigma_{S'}$ mirror dual to $S$, and consider the crepant resolution $\widetilde \cY'$ of the quotient
\[
(E\times S')/\sigma' 
\]
where $\sigma'=(\sigma_E,\sigma_{S'})$. The pair $\widetilde \cY$ and $\widetilde \cY'$ are said to be a Borcea--Voisin mirror pair. (Note that  here we may treat elliptic curves as self-mirror.)  

In a much different way we can define BHK mirror symmetry for certain LG models. An LG model is a pair $(W,G)$ of a quasihomogeneous polynomial $W$ and group of symmetries $G$. Under certain conditions (see Section~\ref{sec:BHK}), we can associate to the pair $(W,G)$ another LG model $(W^T, G^T)$, called the BHK mirror. It has been predicted and partially verified by mathematicians that the A--model construction for $(W,T)$ will be equivalent in a certain way (even in non-CY cases) to the B--model construction for $(W^T,G^T)$. 

Although mirror symmetry makes predictions for equivalence of the A--model and the B--model on many levels, in this paper we focus  on the state space 
of these models. In other words, we consider the A--model state space 
$\sA_{W,G}$, constructed by Fan, Jarvis, and Ruan in \cite{FJR13}---and the B--model state space $\sB_{W^T,G^T}$, given by Saito and Givental in \cite{Sai1, Sai2, Sai3, ST, Giv1, Giv2}. See Section~\ref{sec:LGmodels} for a definition of both models. 
These are known to be isomorphic as vector spaces (see \cite{Kr}), and in some cases as Frobenius algebras (see \cite{FJJS}), though there are still some question regarding the proper definitions of the Frobenius algebra structure. We will return to this question later. 

The LG/CY correspondence predicts a deep relationship between CY orbifolds and LG models in certain cases. On the one side of the LG/CY corresondence, we have the LG model defined by the pair $(W,G)$, and on the CY side, we have the orbifold quotient $X_{W,G}:=[X_W/\widetilde{G}]$ in a quotient of weighted projective space. In this notation, which we use throughout this work, $X_W$ is a hypersurface defined by $W$ and the group $\widetilde{G}$ is simply the quotient of $G$ by the \emph{exponential grading operator}. There are certain conditions on $W$ and $G$, that make $X_{W,G}$ a Calabi--Yau orbifold which we describe in Section~\ref{sec:LGmodels}). On the simplest level, the LG/CY correspondence provides an isomorphism between the Chen-Ruan (orbifold) cohomology $H^*_{CR}(X_{W,G}, \CC)$ (the CY state space) and the FJRW state space $\sA_{W,G}$ mentioned above. This was proven for hypersurfaces by Chiodo--Ruan in \cite{ChR}. 

Somewhat more generally, we may consider quotients of complete intersections $[\{W_1 = \ldots = W_n = 0\}/\widetilde G]$ in products of weighted projective spaces, and relate these to the FJRW theory of $(\sum_i W_i, G)$. It is expected that the correspondence will hold for these complete intersections as well, though it has not been proven. There is some work of Chiodo--Nagel (see \cite{ChN}) and Clader (see \cite{clader}) in this direction.

In particular, the mirror symmetry we have described for CY threefolds is related by the LG/CY correspondence to mirror symmetry for LG models. Using this relationship, we will extend BHK mirror symmetry to certain LG models.  
In this article, we will not be interested in the usual BHK mirror symmetry relating $(W,G)$ and $(W^T, G^T)$, but instead consider the B--model state space for another pair, which we define later. This is the extension of BHK mirror symmetry mentioned earlier, which we will call BVLG mirror symmetry.

The construction of Borcea--Voisin fits into this picture in the following way. First we construct a K3 surface. There is a list of the 95 weight systems $(v_0,\dots,v_3;d_2)$ given by Reid (unpublished) 
and Yonemura in \cite{Yo}, such that the resolution of a quasismooth hypersurface of degree $d_2$ in the appropriate weighted projective space is a K3 surface. Each such weight system satisfies the \emph{Calabi--Yau condition} $d_2=\sum_{i=0}^3 v_i$. We consider the pair $(W,G)$ given by a polynomial of the form 
\begin{equation}\label{e:Wform}
W_2=y_0^2+f(y_1,y_2,y_3) 
\end{equation}
which is quasihomogeneous of degree $d_2$ with respect to this weight system and a group $G_2$ satisfying the same conditions. 
(See Section~\ref{sec:LGmodels} for a definition of these groups.) Of the 95 weight systems, 44 of them admit such a polynomial, namely those weights systems with the property that $d_2=2v_i$ for some $i\in \set{0,\dots, 3}$. We write $X_{W_2,G_2}$ for the quotient by the group $\widetilde{G}_2$ of the hypersurface defined by the vanishing of $W_2$ in the appropriate quotient of weighted projective space. There is an obvious nonsymplectic involution on $X_{W_2,G_2}$ namely $\sigma_2:y_0\to -y_0$. The minimal resolution of $X_{W_2,G_2}$ is a K3 surface with an involution induced by $\sigma_2$. 

We construct an elliptic curve in a similar way. In other words, we can consider polynomials of the form $W_1=x_0^2+f_1(x_1,x_2)$ quasihomogeneous with respect to one of the weight systems $(u_0,u_1,u_2;d_1)=(3,2,1;6)$ or $(2,1,1;4)$, and a group of diagonal symmetries G satisfying the same conditions mentioned previously.  
The resolution of $X_{W_1,G_1}$ will be an elliptic curve and has an obvious corresponding involution. 

Setting $\sigma=(\sigma_1,\sigma_2)$ we call the quotient stack 
\[
\Big[(X_{W_1,G_1}\times X_{W_2,G_2})/\sigma \Big]
\]
a \emph{Borcea--Voisin orbifold}. As mentioned above, this is a complete intersection in the quotient of a product of weighted projective spaces. This orbifold has a crepant resolution, which is a Borcea--Voisin threefold. One expects the LG/CY correspondence to hold for this model as well, however it falls outside of the work of Chiodo--Ruan \cite{ChR} and Chiodo--Nagel \cite{ChN}. We will return to the question of the LG/CY correspondence for this case in Section~\ref{s:geometry}, where we provide a state space isomorphim between the CY state space and the FJRW state space.

Inspired by the LG/CY correspondence, we will describe the corresponding LG model, which we will see (in Section~\ref{ss:GLSM} is given by the sum $W=W_1+W_2$ and the group generated by $\sigma$ and $G_1\times G_2$. We denote this group by $\s{(G_1\times G_2)}$. Thus, our LG model is the pair $(W, \s{(G_1\times G_2)})$. In fact, we can generalize this construction on both sides of the LG/CY correspondence to more arbitrary symmetry groups and to any dimension.

BV mirror symmetry requires a mirror K3 surface to construct the mirror manifold for a given BV threefold. In \cite{involutions}, Artebani--Boissi\`ere--Sarti showed that the mirror symmetry described by Nikulin for K3 surfaces with involution agrees with what one would expect to obtain from BHK mirror symmetry for the pair $(W_2,G_2)$. Thus we can also describe a Borcea--Voisin mirror pair by constructing the K3 surface defined by $(W_2^T, G_2^T)$.

Interestingly, the LG model provided by BHK mirror symmetry does not yield the same mirror as one would expect from Borcea--Voisin mirror symmetry. In this article we suggest another formulation of LG to LG mirror symmetry, which matches the Borcea--Voisin mirror symmetry construction and extends BHK mirror symmetry to this setting. We do this by defining the pair $(W^T,\s{(G_1^T\times G_2^T)})$ as the Borcea--Voisin LG mirror dual to the pair $(W, \s{(G_1\times G_2)})$.

The main result of this work justifies calling these a mirror pair, giving an isomorphism of state spaces. We state this result as follows:  

\begin{customthm}{1}\label{thm1}
There is a bi-degree preserving vector space isomorphism $\sA_{W,\s{G}}\cong \sB_{W^T,\s{G^T}}$, relating the A--model state space and the B--model state space for BVLG mirror pairs $(W,\s{G})$ and $(W^T, \s{G^T})$. 
\end{customthm}

On the way to proving this theorem, we obtain an isomorphism (called the twist map) within FJRW theory, which facilitates several computations of interest in the Landau-Ginzburg mirror symmetry arena. This is similar to the geometric \emph{twist map} described by Borcea in \cite{Borcea}.

In Section~\ref{sec:LGmodels}, we describe the construction of each state space for the LG models, and in Section~\ref{sec:twist} we discuss the LG twist map and prove the main theorem. In Section~\ref{s:geometry} we prove several interesting geometric consequences of the main theorem. 

In some sense these theorems are LG analogues of results in both the work of Artebani--Boissi\`ere--Sarti in \cite{ABS} and of Chiodo--Kalashnikov--Veniani in \cite{CKV} simultaneously, which we discuss below. Other than the examples computed by the last author in \cite{Schaug} this is the first time we are aware that this particular LG model has been considered in the literature. 

The state space of FJRW theory possesses an interesting algebraic structure as mentioned earlier; it can be given the structure of a Frobenius algebra defined in \cite{FJR13}, which is supercommutative. The LG B--model state space can also be given such a structure. In Section~\ref{sec:alg_isom}, we also extend the isomorphism of state spaces to an isomorphism of Frobenius algebras on a certain subspace of the LG state spaces, namely, the even-graded subspace. 

Up to this point, to our knowledge, no one has yet considered the $\ZZ_2$ grading, and so it was thought that the A--model state space was commutative. On the other side, the B--model has several possible definitions, including that of Krawitz \cite{Kr}, Basalaev--Takahashi--Werner \cite{BTW}, and He--Li--Li \cite{HLL}. The definition of Krawitz does not have any supercommutativity, whereas the other two do. However, if we restrict to even graded subspaces, the supercommutativity becomes simply commutativity. 

In this article, we restrict ourselves to even-graded subspace of the state space in order to avoid complications arising from supercommutativity. It would be interesting in the future to compare the various B--models.

\subsection{Related Work} The Borcea--Voisin construction has been the topic of several recent articles, which are related to the current work. We briefly review those results.  

In \cite{ABS}, Artebani--Boissi\`ere--Sarti prove a theorem in which they relate BHK mirror symmetry and BV mirror for the same polynomials considered here. There they consider only geometric models, using the LG model as a guide. They prove that the Calabi--Yau threefolds of Borcea--Voisin are birationally equivalent to the image of the twist map of Borcea (see \cite{Borcea} and Section~\ref{ss:twistmap}). 

Their result relies on the assumption that $\gcd(u_0,v_0)=1$. In order to understand this restriction better, consider that there are two possible weight systems for $W_1$ yielding an elliptic curve and 44 possible weight systems for $W_2$ yielding a K3 surface with involution. Recall in this construction, we require our polynomials to be of the form \eqref{e:Wform}. Only 48 of the 88 possible combinations of weight systems satisfy the gcd condition imposed in \cite{ABS}. In this article we generalize this result in two ways. We remove the restriction on gcd's, and we extend the construction to all dimensions. 

In \cite{CKV} Chiodo--Kalashnikov--Veniani undertake to prove a mirror statement for the geometric model (i.e. for orbifold BV models), again using the LG model as a guide. There, however, the goal is to provide a mirror map for BV orbifolds. In other words, they prove an isomorphism of Chen--Ruan cohomology, similar to Theorem~\ref{cor:BVMS} in the current work. Chiodo--Kalashnikov--Veniani consider two separate LG models, one for $W_1$ and one for $W_2$, and then use a K\"unneth type formula to prove their result. Here we also generalize this result with a different method of proof, removing the restriction that $G$ be a product of groups, which was necessary in their work. When $G$ is not a product of groups, the crepant resolution of the quotient
\[
X_{W_1}\times X_{W_2}/\widetilde{\s{G}}
\] 
is no longer a Borcea--Voisin threefold, but rather is the Nakamura--Hilbert scheme of $\widetilde{\s{G}}$--orbits of $X_{W_1}\times X_{W_2}$.

In \cite{Schaug2}, the last author has considered exactly the form of mirror symmetry we propose here with the restriction that the defining polynomials must be Fermat type. In fact, he was able to show that for the mirror pairs we consider here, there is a mirror map relating the FJRW invariants of the A--model to the Picard--Fuchs equations of the B--model. In \cite{Schaug} he also gave an LG/CY correspondence relating the FJRW invariants of the pair $(W,\s{G})$ to the corresponding Gromov--Witten invariants of the corresponding Borcea--Voisin orbifold. Although these results are broader in scope, the restriction to Fermat type polynomials is  significant, reducing the number of weight systems from which one can select a K3 surface to 10 (from the 48 mentioned above). Furthermore, there is no general method of proof for a state space isomorphism provided there. However, we expect results regarding the FJRW invariants, Picard--Fuchs equations and GW invariants to hold in general, and the state space isomorphism we establish here is the first step to such results. This will be the topic of future work. 

Computational observations tell us that in some cases a pair $(W,G)$ may have multiple mirrors in BHK mirror symetry.  For example, Shoemaker \cite{shoemaker} and Kelly \cite{kelly} investigated the birationality of multiple mirrors obtained via the so-called weights and groups theorem. But these do not account for all possible multiple mirrors.  Our results give a new algebraic observation about possible multiple mirrors in the Landau-Ginzburg mirror symmetry arena.

Finally, an extension to automorphisms of higher order on the geometric side is considered by Goto--Kloosterman--Yui in \cite{GKY}. It will be interesting to see if the mirror symmetry of Borcea--Voisin can be extended to such examples as well on the LG side as on the geometric side. Goto--Livn\'e--Yui \cite{GLY} also consider Borcea--Voisin Calabi--Yau threefolds of the type studied in the current article, and study how the L-series behaves under mirror symmetry. 

\section*{Acknowledgements} The authors would like to thank Alessandra Sarti, Alessandro Chiodo, Davide Veniani, Yongbin Ruan and our referee for helpful comments on earlier drafts. The second author would also like to thank the Institut f\"ur Algebraische Geometrie in the Leibniz Universit\"at Hannover, where the initial work for this publication was completed.

\section{Landau--Ginzburg models}\label{sec:LGmodels}
We begin by reviewing the constructions of the LG model for a pair $(W,G)$.

A polynomial 
$
W= \sum_{i=1}^m c_i\prod_{j = 1}^n x_j^{a_{i,j}}  \in \mathbb{C}[x_1, \ldots, x_n]
$
is called \emph{quasihomogeneous} if there exist positive weights $q_j=\tfrac{w_j}{d}$ for all  $j = 1, \ldots, n$ such that every nonzero monomial of $W$ has weighted degree one. We require $\gcd(w_1,\dots,w_n)=1$, and we call $(w_1,\dots,w_n;d)$ a \emph{weight system}.

Additionally, a polynomial $W \in \mathbb{C}[x_1, \ldots, x_n]$ is called \emph{nondegenerate} if it has an isolated singularity at the origin and the weights $q_i$ are uniquely determined. In this case we say that $W$ is \emph{admissible}. 
An admissible polynomial is called \emph{invertible} if the number of variables is equal to the number of monomials in the polynomial. 

The Landau--Ginzburg model also requires a choice of symmetry group. 
For an admissible polynomial $W$, the \emph{maximal diagonal symmetry group} $G_W^{max}$ is the group of elements of the form $g = (g_1, \ldots, g_n) \in \left(\mathbb{Q}/\ZZ\right)^n$ such that
\[
W(e^{2\pi i g_1}x_1,e^{2\pi i g_2}x_2,\ldots,e^{2\pi i g_n}x_n) = W(x_1,x_2,\ldots,x_n).
\]
Notice that if $q_1, q_2, \ldots, q_n$ are the weights of $W$, then the \emph{exponential grading operator} $\jw_W = (q_1, \ldots, q_n)$ is an element of $G_W^{max}$.  

An \emph{A--admissible symmetry group} for $W$ is a subgroup $G\subset G_W^{max}$ which contains $\jw_W$.   We denote by $J_W=\left\langle \jw_W\right\rangle$ the subgroup generated by $\jw_W$, and thus any subgroup between $J_W$ and $G_W^{max}$ is A--admissible (see \cite{Kr}). 

We can also embed $G_W^{max}$ into $\GL_n(\CC)$ via the the diagonal matrices
\[
(g_1,\dots,g_n)\mapsto 
\left(\begin{matrix}
	e^{2\pi i g_1} & & 0 \\
	& \ddots & \\
	0& & e^{2\pi i g_n}
\end{matrix}\right) 
\]
and define the group $\SL_W=G_W^{max}\cap \SL_n(\CC)$. Equivalently, for $g=(g_1,\dots,g_n)\in G^{max}_W$, we can consider $g_i$ simply as rational numbers with $0\leq g_i< 1$. We define the \emph{age} of $g$ to be 
\[
 \age g = \sum_{j=1}^n g_i. 
\]
Then $\SL_W$ is the subgroup of elements $g\in G^{max}_W$ with $\age g\in \ZZ$. A group $G\subseteq G^{max}_{W}$ is \emph{B--admissible} if $G\subset \SL_{W}$. We will generally be interested in A--admissible groups one polynomial $W$ and B--admissible groups which are subgroups of the maximal symmetry group for a different polynomial, which we will call $W^T$, which we define in the next section.

We say that a nondegenerate quasihomogeneous polynomial $W$ satisfies the Calabi--Yau condition if $\age \jw_W=1$. Notice that this condition ensures that $J_W\subset \SL_W$. The following constructions of the A-- and B--models do not require this condition. But we will require this condition beginning with Section~\ref{sec:twist}.

\subsection{Construction of A-- and B--models}
We now briefly review the construction of the A-- and B--model state spaces. Each of these yields a graded vector space with a nondegenerate pairing. The pairing will only be important in the Frobenius structure, so we leave it to Section~\ref{sec:alg_isom}.

Recall that we consider $G^{max}_W$ as a subgroup of $\GL_n(\CC)$. For $g\in G^{max}_W$, we fix the following notation:
\begin{itemize}
\item $\Fix(g)=\set{ \mb{x}\in \CC^n\mid g\mb{x}=\mb{x}}$, 
\item $W_g=W|_{\Fix(g)}$, 
\item $N_g$ the dimension of $\Fix(g)$,
\item $I_g = \{i \hspace{.1cm} |\hspace{.1cm}  g \cdot x_i = x_i\}$ 
\end{itemize}
Note that $I_g$ is the set of indices of those variables fixed by $g$, and that $N_g=|I_g|$.

\subsubsection*{The Landau--Ginzburg A--model.}\label{ss:LGA} The A--model was first constructed by Fan--Jarvis--Ruan in \cite{FJR13} following ideas of Witten, and has since come to be known as FJRW theory. For $W$ a nondegenerate quasihomogeneous polynomial and $G$ an A--admissible group, the theory yields a state space, a moduli space of so--called $W$-curves, and a virtual cycle from which one can define numerical invariants. In this work, we are primarily concerned with the state space $\sA_{W,G}$, which is defined in terms of Lefschetz thimbles.

We define $\sA_{W,G}$, as 
\begin{equation}\label{e:Aspace}
\sA_{W,G} = \bigoplus_{g \in G} H^{mid} (\Fix(g),(W)_g^{-1}( \infty))^G,
\end{equation}
where $W_g^{-1}(\infty)$ is a generic smooth fiber of the restriction of $W$ to $\Fix(g)$.

Because Lefschetz thimbles are more difficult to work with, we give an alternate defintion of the A--model state space, which will facilitate computation. 

The \emph{Milnor Ring} $\sQ_W$ of the polynomial $W$ is defined as 
\[
\sQ_W =\CC[x_1 \ldots, x_n]/
\left\langle\left\{\partial W/\partial x_i\right\}\right\rangle.
\]
If $W$ is nondegenerate, then $\sQ_W$ has finite dimension as a $\CC$-vector space. We denote $\mu_W:=\dim_\CC \sQ_W$. 

The following theorem gives a connection between the middle dimensional cohomology and the Milnor ring (also see \cite{ChIR} where the isomorphism is given canonically). 

\begin{thm}[Wall \cite{Wall1}, Chiodo--Iritani--Ruan \cite{ChIR}]\label{t:Wall}
Let $\omega = dx_1 \wedge \ldots \wedge dx_n$, then 
\[
H^{mid} (\CC^n,(W)^{-1}( \infty)) \cong  \frac{ \Omega^n}{dW \wedge\Omega^{n-1}} \cong \sQ_W\cdot \omega,
\]
and this isomorphism respects the $G_W^{max}$--action on both. 
\end{thm}

In this description of the state space, $G_W^{max}$ acts on both monomials in $\sQ_W$ as well as the volume form $\omega$.
The isomorphism in Theorem~\ref{t:Wall} certainly will hold for the restricted polynomials $W_g$ as well.  Therefore, setting $\sQ_g=\sQ_{W_g}$, we can rewrite the state space canonically as 
\begin{equation}\label{state_space}
\sA_{W,G} = \bigoplus_{g \in G} \left(\sQ_g \omega_g \right)^G,
\end{equation}
where $\omega_g = dx_{i_1} \wedge \ldots \wedge dx_{i_s}$ for $i_j \in I_g$.

\begin{notat}
The state space $\sA_{W,G}$ is a $\CC$--vector space. We can construct a basis out of elements of the form $\fjrw{m}{g}$, where $I_g = \{i_1, \ldots, i_r\}$ and $m$ is a monomial in $\CC[x_{i_1}, \ldots, x_{i_r}]$ multiplied by the volume form $\omega_g$.  We say that $\fjrw{m}{g}$ is \emph{narrow} if $ I_g = \emptyset$ (in which case $m=1$), and \emph{broad} otherwise.
\end{notat}

The A--model also has a bigrading, defined as
\begin{equation}\label{e:fjrwbidegree}
(\deg^A_+(\fjrw{m}{g}), \deg^A_-(\fjrw{m}{g})) = \left(\deg m+\age g-\age\jw_W,\ N_g-\deg m+\age g-\age\jw_W\right).
\end{equation}

\subsubsection*{The Landau--Ginzburg B--model.}
The B--model state space and its pairing are defined similarly to the A--model, with subtle but important differences. 
The input data is a pair $(W, G)$ with $W$ a nondegenerate quasihomogeneous polynomial and $G$ a B--admissible symmetry group. The state space is defined as 
\begin{equation*}\label{e:Bstatespace}
\sB_{W,G} = \bigoplus_{g \in G} \left(\sQ_g \omega_g \right)^G.
\end{equation*}

The bi-degree of $\fjrw{m}{g}$ in the B--model is 
\[
(\deg^B_+(\fjrw{m}{g}), \deg^B_-(\fjrw{m}{g})) = \left(\deg m+\age g-\age\jw_W,\ \deg m+\age g^{-1}-\age\jw_W\right).
\] 

\begin{rem}\label{r:ABiso}
As vector spaces 
we have an isomorphism $\sA_{W,G}\cong \sB_{W,G}$ (see Equation \ref{state_space}). This fact will be useful later. However, this isomorphism does not preserve the bigrading. Furthermore, both state spaces can be given the structure of $\CC$--algebras (see Section~\ref{sec:alg_isom}), and, in general, are not isomorphic as $\CC$-algebras. For more general quasihomogeneous polynomials, most symmetry groups are not both A--admissible and B--admissible as we have here, so the isomorphism just mentioned will make no sense, as the A--model only allows  A--admissible groups and likewise for the B--model. 
\end{rem}

\subsection{Berglund--Hubsch--Krawitz Mirror Symmetry}\label{sec:BHK}
There is a more interesting relationship between A-- and B--models which was first discovered by Berglund--H\"ubsch in \cite{berghub} and later refined by Berglund--Henningson in \cite{BHen} and Krawitz in \cite{Kr}. This form of mirror symmetry relates the A--model of the pair $(W,G)$ to the B--model of another pair $(W^T, G^T)$, which we will now define. We will hereafter refer to this formulation of mirror symmetry as BHK mirror symmetry.  

Suppose $W=\sum_{i=1}^n  \prod_{j=1}^n x_j^{a_{i,j}}$ is an invertible polynomial. We can suppress any coefficients via a change of variables. The \emph{mirror potential} $W^T$ is 
\[
W^T = \sum_{i=1}^n  \prod_{j=1}^n x_j^{a_{j,i}}. 
\]
In other words, if $A_W$ is the exponent matrix for $W$, then $W^T$ is the polynomial corresponding to the exponent matrix $A_W^T$. One can check that $W^T$ is also an invertible polynomial (see \cite{ChR}). 
 
If $G$ is an admissible symmetry group for a polynomial $W$, then the \emph{mirror group} $G^T$ is defined by 
\[
G^T = \set{\; g \in G_{W^T}^{max} \; | \;\;g A_{W} h^T \in \ZZ \text{ for all } h \in G \; }.
\]
\begin{rem}
This is not the original definition of $G^T$. The first definition was given by Berglund--Henningson in \cite{BHen}, and then another definition was given independently by Krawitz in \cite{Kr}. In \cite{involutions}, it is shown that their defintions agree, and it is an exercise to check that the definition given here is equivalent to the one given by Krawitz. 
\end{rem}

In \cite[Proposition 3]{involutions} we learn that $(G^T)^T = G$ and that $J_W\subset G$ if and only if $G^T\subset \SL_W$. In other words, this duality exchanges A--admissible groups for B--admissible ones. 

On the level of state spaces, we get an isomorphism between A-- and B--models given by the following theorem of Krawitz.

\begin{thm}[Krawitz \cite{Kr}]\label{t:krawitz}
Let $W$ be a non--degenerate invertible potential and $G$ an A--admissible group of diagonal symmetries of $W$.
There is a bigraded isomorphism of vector spaces
\[
\sA_{W,G} \cong \sB_{W^T ,G^T},
\]
where $\sA_{W,G}$ is the FJRW A--model of $(W, G)$ and $\sB_{W^T ,G^T}$ is the orbifold B--model of $(W^T, G^T)$.
\end{thm}

One may notice that Theorem~\ref{thm1} is similar to Theorem~\ref{t:krawitz}, though they differ in the choice of symmetry group. However, Theorem~\ref{t:krawitz} will be used in the proof of Theorem~\ref{thm1}.

\begin{rem}
The proof of Theorem~\ref{t:krawitz} in \cite{Kr} does not hold in case $W$ contains a summand of a so--called chain polynomial with weight $\tfrac 12$. However, such a polynomial will not appear in the current work. 
\end{rem}

\subsection{Mirror Pairs inspired by Borcea--Voisin mirror pairs}\label{ss:BVLGmodel}
Following ideas inspired by the LG--CY correspondence, we can define a Landau--Ginzburg model that will correspond to the Calabi--Yau threefolds constructed by Borcea and Voisin. 

The LG--model corresponding to Borcea--Voisin threefolds can be described in many cases as follows. Consider two invertible polynomials of the form 
\begin{equation}\label{e:wxwy}
W_1=x_0^2+f_1(x_1,\dots,x_n)\quad \text{ and }\quad W_2= y_0^2+f_2(y_1,\dots,y_m), 
\end{equation}
with weights $(u_0,\dots, u_n;d_1)$ and $(v_0,\dots,v_m;d_2)$ Notice that because of the special form of the polynomials, $d_1 = 2u_0$ and $d_2 = 2v_0$. In order to ease notation slightly, in what follows we will denote $J_i=J_{W_i}$ generated by the exponential grading operator $\jw_i=\jw_{W_i}$ for $i=1,2$.  
 
In the Borcea--Voisin construction we take the product of hypersurfaces defined by $W_1$ and $W_2$, respectively, then form the quotient by a certain involution. On the LG--side we consider the ``twisted intersection'' of these two polynomials. That is, let $W=W_1+W_2$ and let $G$ be any group satisfying $J_1\times J_2\leq G\leq \SL_{W_1}\times \SL_{W_2}$. Denote the involution 
$\sigma=(1/2,\mb{0}_x,1/2,\mb{0}_y)$, and define the group $\s{G}=G\cup \sigma G$---that is, the smallest group that contains $G$ and $\sigma$. This involution is the LG analogue of the involution in the Borcea--Voisin model. We call the LG model $(W,\s{G})$ a Borcea--Voisin Landau--Ginzburg (BVLG) model. We will justify this construction and the use of this terminology in Section~\ref{ss:GLSM}. 

Following Borcea--Voisin mirror symmetry described earlier, we call the pair $(W,\s{G})$ and $(W^T,\s{G^T})$ a BV mirror pair. As already mentioned, BV mirror symmetry does not actually match the picture given by BHK mirror symmetry, as one might expect. BHK mirror symmetry defines a group $(\s{G})^T$, which does not contain $J_1\times J_2$, but $\s G^T$ does (e.g. since $\s{G}\nsubseteq \SL_{W_1}\times \SL_{W_2}$). So the BHK mirror of a BVLG model is no longer a BVLG model.

\begin{rem}
What we call a BVLG model is slightly more general than the original Borcea--Voisin construction. If we take for $G$ a group that can be written as a product of subgroups $G=G_1\times G_2$ with $G_1\subset \SL_{W_1}$ and $ G_2\subset \SL_{W_2}$, then we recover the construction of Borcea--Voisin. If, on the other hand, one considers a group $G$ that is not such a product, the corresponding geometry is no longer a variety of Borcea--Voisin type, but rather an orbifold, whose crepant resolution is the Nakamura--Hilbert scheme of $\widetilde{\s G}$-regular orbits of $X_1\times X_2$ (see \cite{Na}), as was also mentioned in the introduction.  
\end{rem}

\section{The Twist Map}\label{sec:twist}
In this section we make some preparations needed for the proofs of Theorem~\ref{thm1}. 
From this point on, we will assume the Calabi--Yau condition on our weight systems (see Section~\ref{sec:LGmodels}). We will use ideas from the geometric model, the most important being the twist map from \cite{Borcea}. There is a similar notion on the LG side, which we will call the \emph{LG twist map}. 

Given a pair $(W,\s{G})$ as described above, namely $W=W_1+W_2$ with $W_1$ and $W_2$ of the form \eqref{e:wxwy}, and $G$ with $J_1\times J_2\subset G\subset \SL_{W_1}\times \SL_{W_2}$, we define another LG model $(\tw{W}, \tw{G})$ corresponding to the geometry of the image of the twist map (see Section~\ref{ss:twistmap} and \cite{ABS}).

We first define $\tw{W}=f_1-f_2$. If $\delta=\gcd(u_0,v_0)$, then this defines an invertible polynomial with weight system $(\tfrac{v_0u_1}{\delta},\tfrac{v_0u_2}{\delta},\dots,\tfrac{v_0u_n}{\delta},\tfrac{u_0v_1}{\delta},\dots,\tfrac{u_0v_m}{\delta};\lcm(d_1,d_2))$. We use the negative sign as a convention for clarity, in order to match the geometry. Since $f_2$ is invertible, this sign convention can be scaled away.

To see that $\tw{W}$ is quasihomogeneous, note that $\lcm(d_1,d_2)=2\lcm(u_0,v_0)$, and the same for the $\gcd$'s. Further,  $\lcm(d_1,d_2)\cdot\gcd(d_1,d_2)=d_1d_2$, so  $\delta\cdot\lcm(d_1,d_2)=\tfrac{d_1d_2}{2}$ and $v_0=\tfrac{d_1}{2}$. Hence
\[
q_{x_i}=\tfrac{u_i}{d_1}=\tfrac{v_0u_i}{\delta}\tfrac{1}{\cdot\lcm(d_1,d_2)}.
\]
Furthermore, the gcd of all of the weights is equal to 1.

To define the group $\tw G$, notice that for any $g'\in G^{\max}_{\tw{W}}$,
we can write $g'=(\alpha, \beta)$, where $\alpha$ acts on the variables $x_1, \ldots x_n$, and $\beta$ on the variables $y_1, \ldots, y_m$. There is an injective group homomorphism $\phi:G_{\tw{W}}^{\max}\hookrightarrow G_{W}^{\max}$, given by 
\[
\phi: (\alpha,\beta)\mapsto (0,\alpha, 0, \beta). 
\]
We define $\tw{G}$ to be $\phi^{-1}(\s{G})$. If $\delta=1$, then this is the group considered by Artebani--Boissi\`ere--Sarti in \cite{ABS}. Otherwise, the group is larger by a factor of $\delta$. 

The proof of Theorem~\ref{thm1} will rely heavily on the twist map, as suggested in the following diagram.

\begin{center}\begin{tabular}{ccc}
	$\sB_{\tw{(W^T)},\tw{(G^T)}}$&\begin{tikzpicture} \draw[<-] (0,0) -- (1,0); \node at (.5,.25) {$tw_B$}; \end{tikzpicture} &$\sB_{W^T, \s{G^T}}$ \\
	\begin{tikzpicture} \draw[<->] (0,0) -- (0,1); \end{tikzpicture}& & \begin{tikzpicture} \draw[<->, dashed] (0,0) -- (0,1); \node at (0,1.1) {}; \end{tikzpicture} \\
	$\sA_{\tw{W},\tw{G}}$ &\begin{tikzpicture} \draw[<-] (0,0) -- (1,0); \node at (.5,.25) {$tw_A$};\end{tikzpicture} &$\sA_{W, \s{G}}$ \\
\end{tabular}\end{center}

Now we consider the relationship between $\sA_{\tw{W},\tw{G}}$ and $\sA_{W,\s{G}}$. In \cite{Borcea, ABS} Borcea and Artebani--Boissi\`ere--Sarti have shown that when $\delta=1$ the corresponding geometric models are birational. We give a similar statement for the LG--models.

\begin{thm}\label{t:twist}
There is an isomorphism of vector spaces $tw_A:\sA_{W,\s{G}}\to \sA_{ \tw{W}, \tw{G}}$ which preserves bi-degrees. 
\end{thm}

This theorem is proved by giving a bijection on a basis of both spaces, and showing it preserves bi-degrees.  
We first prove the following preparatory lemma. 

\begin{lem}\label{l:allornothing}
If $(\sQ_{g})^{\s{G}}$ is a non--empty sector of $\sA_{W,\s{G}}$, then $g$ either fixes both of $x_0$ and $y_0$ or neither. 
\end{lem}
\begin{proof}
Suppose $g$ fixes either $x_0$ or $y_0$ but not both. Then $g$ has the form $g=(1/2,\alpha, 0,\beta)$ or $g=(0,\alpha, 1/2, \beta)$. In either case the volume form $\omega_g$ includes $dx_0$ or $dy_0$, but not both. Then $\sigma$ acts on any element of the Milnor ring $\sQ_{W_g}$ with weight $1/2$. Therefore there are no invariants. 
\end{proof}

\begin{proof}[Proof of Theorem~\ref{t:twist}] 
Consider $\fjrw{p}{g}\in \sA_{W,\s{G}}$. By Lemma~\ref{l:allornothing}, we can write $g=(\epsilon/2, \alpha, \epsilon/2, \beta)$, $\epsilon\in \set{0,1}$, and $p=dx_0^{1-\epsilon} dy_0^{1-\epsilon}\prod x_j^{a_j}dx_j \prod y_j^{b_j}dy_j$. Here the products are taken over $I_\alpha$ and $I_\beta$, resp. (we can regard $\alpha$ and $\beta$ as symmetries of $f_1$ and $f_2$, resp.). We define the map via
\[
tw_A(\fjrw{p}{g})=\fjrw{p'}{g'}, 
\]
where 
\[
p'=2^{(1-\epsilon)}\prod_{j\in I_\alpha} x_j^{a_j}dx_j \prod_{j\in I_\beta} y_j^{b_j}dy_j,
\] 
and $g'=(\alpha,\beta)$. The factor of 2 in the definition of the twist map exists to preserve the pairing, which we will return to in Section~\ref{sec:alg_isom}. We check that this map is well defined, injective and surjective, and that it preserves the bigradings. 

To see that the map is well defined, we check that $g'\in \tw{G}$, and that $p'$ is fixed by $\tw{G}$. For the first, notice that either $g=(0,\alpha,0,\beta)$ or $\sigma g=(0,\alpha,0,\beta)$. Since $(0,\alpha,0,\beta)\in \s{G}\cap \phi (G^{\max}_{\tw{W}})$ we see that $g'\in \tw{G}$. For the second, note that $p'$ is fixed by $\tw{G}$ because $\phi$ restricts to an injection $\phi:\tw{G}\hookrightarrow \s{G}$, and $g'$ acts on $p'$ exactly the same as $\phi(g')\in \s{G}$ acts on $p$. 

As for the surjectivity, in the notation already in use, suppose $\fjrw{p'}{g'}\in \sA_{\tw{W},\tw{G}}$. Let $p_x=\prod x_j^{a_j}dx_j$ and $p_y=\prod y_j^{b_j}dy_j$ (so, $p'= 2^{1-\epsilon}p_xp_y$). Again, the products are taken only over the respective sets of fixed variables. Let $\jw_{1}$, $\jw_{2}$ and $\tw{\jw}$ be the exponential grading operators for $W_1$, $W_2$, and $\tw{W}$, resp. When there is no danger of confusion, we will use the same notation for the elements in $\tw{G}$ and $\phi(\tw{G})$. For example, if $\tw{\jw}=(\jw_x,\jw_y)$, then we may also write $\tw{\jw}=(0,\jw_x,0,\jw_y)$. Notice that $\jw_x$ and $\jw_y$ are the exponential grading operators for $f_1$ and $f_2$, resp. The context should make clear which group is meant. Notice we also have $\jw_{1}=(1/2,\jw_x,0,0)$, and $\jw_{2}=(0,0,1/2,\jw_y)$.  

In the notation just described, $2\jw_{1}=(0,2\jw_x,0,0)$, and so $(2\jw_x,0)\in \tw{G}$. Since $p'$ is fixed by $\tw{G}$, it is fixed by $(2\jw_x,0)$, which acts trivially on the $y$ variables, and so $(2\jw_x,0)$ fixes $p_x$. 
Hence $\jw_x$ either fixes or acts with weight $1/2$ on $p_x$. The same for $\jw_y$ on $p_y$. And since ${\tw{\jw}}$ fixes $p'$, the action is the same for both. 

If $\jw_x$ fixes $p'$, then the preimage of $\fjrw{p'}{g'}$ is $\fjrw{p}{g}$ with $g=(1/2,\alpha,1/2,\beta)$ and $$p=\prod x_j^{a_j}dx_j \prod y_j^{b_j}dy_j = p_x p_y.$$ Notice that the action of any $(\gamma_x, \gamma_y) \in \tw{G}$ on $p'$ is the same as the action of $(\alpha_0,\gamma_x, \beta_0, \gamma_y)\in \s{G}$ on $p$ for any $\alpha_0=\beta_0$. There are also elements of $\s{G}$ of the form $(1/2,\alpha,0,\beta)$ or $(0,\alpha,1/2,\beta)$, with $(\alpha,\beta)\notin \tw{G}$. Since the first is simply $\sigma$ times the second, we can focus only on the first. In this case, we can multiply by $\jw_{1}$ and we see that $(\jw_x\alpha,\beta) \in \tw{G}$. Since $\jw_x$ fixes $p'$, we see that $(1/2,\alpha,0,\beta)$ must fix $p$. Thus, $p$ is fixed by $\s{G}$ and $\fjrw{p}{g}$  is indeed an element of $\sA_{W,\s{G}}$.

If on the other hand, $\jw_x$ acts with weight $1/2$ on $p_x$, then the preimage of $\fjrw{p'}{g'}$ is $\fjrw{p}{g}$ with $g=(0,\alpha,0,\beta)$ and $p=dx_0 dy_0\prod x_j^{a_j}dx_j \prod y_j^{b_j}dy_j=dx_0 dy_0p_xp_y$. Indeed, it is clear that $g\in \s{G}$, we need only check that $p$ is fixed by $\s{G}$. Because $p'$ is fixed by $\tw{G}$, we see that $p$ is fixed by all elements $h$ of $\s{G}$ which have the form $(0,h_x,0,h_y)$ or $(1/2,h_x,1/2,h_y)$. We need only further consider the case when $h=(1/2,\eta_x,0,\eta_y)\in \s{G}$. Then we have $\jw_{1}+h=(0,\jw_x+\eta_x,0,\eta_y)\in \s{G}$, and therefore $\jw_{1}+h=(\jw_x+\eta_x,\eta_y)\in \tw{G}$ fixes $p'$. The action of $\jw_x$ on $p'$ has weight $1/2$, by assumption, which means the action of $(\eta_x, \eta_y)$ on $p'$ must also have weight $1/2$. Because $h$ acting on $dx_0$ contributes another $1/2$, we see that $h$ fixes $p$. 

For the injectivity of the twist map, we need to show that if $\fjrw{p}{g}\in \sA_{W,\s{G}}$ with $g=(0, \alpha, 0, \beta)$, then $\fjrw{\tfrac{p}{dx_0 dy_0}}{\sigma g}\notin \sA_{W,\s{G}}$, and vice--versa. This follows by similar reasoning. For example, if $\jw_{1}$ fixes $p$, then it acts on $\tfrac{p}{dx_0 dy_0}$ with weight $1/2$. Similarly, if  $\jw_{1}$ fixes $\tfrac{p}{dx_0 dy_0}$, then it acts on $p$ with weight $1/2$. Thus both cannot belong to $\sA_{W,\s{G}}$. 

Now we show that this map preserves degrees.  We will consider the bigradings of the elements $\fjrw{p}{g}$ with $g=(\epsilon/2,\alpha,\epsilon/2,\beta)$ and $\fjrw{p'}{g'}$ with $g=(\alpha, \beta)$. And indeed,  
\begin{align*}
\deg^A_+(\fjrw{p}{g}) &= \deg p + \age g - 2\\  
    &= (\tfrac{1-\epsilon}{2}+\tfrac{1-\epsilon}{2})+\deg p'+\tfrac{\epsilon}{2}+\tfrac{\epsilon}{2}+\age g'-2 \\
    &= \deg p'+\age g'-1\\
    &= \deg^A_+(\fjrw{p'}{g'})
\end{align*}
and since for $g=(\epsilon/2,\alpha,\epsilon/2,\beta)$, we have $N_g=2(1-\epsilon)+N_{g'}$, we also have 
\begin{align*}
\deg^A_-(\fjrw{p}{g}) &= N_g-\deg p + \age g -2\\ 
    &= (2(1-\epsilon)+ N_{g'})-(\tfrac{1-\epsilon}{2}+\tfrac{1-\epsilon}{2}+\deg p')+(\tfrac{\epsilon}{2}+\tfrac{\epsilon}{2}+\age g')-2 \\
    &= N_{g'}-\deg p'+\age g'-1\\
    &= \deg^A_-(\fjrw{p'}{g'})
\end{align*}

Thus the bidegrees are preserved, and this concludes the proof. 
\end{proof}

\begin{cor}\label{c:Btwistiso}
There is an isomorphism of vector spaces $tw_B:\sB_{W,\s{G}}\to \sB_{\tw{W}, \tw{G}}$ which preserves the bi-degrees. 
\end{cor}

\begin{proof}
This isomorphism $tw_B$ is defined as in Theorem~\ref{t:twist}. Recall that $\sB_{W,G}\cong \sA_{W,G}$ as vector spaces 
(see Remark \ref{r:ABiso}). The only thing that remains to check is the bi--degree. 

For $g=(\epsilon/2,\alpha,\epsilon/2,\beta)$, 
\begin{align*}
\deg^B_+(\fjrw{p}{g}) &= \deg p + \age g - 2\\  
    &= (\tfrac{1-\epsilon}{2}+\tfrac{1-\epsilon}{2})+\deg p'+\tfrac{\epsilon}{2}+\tfrac{\epsilon}{2}+\age g'-2 \\
    &= \deg p'+\age g'-1\\
    &= \deg^B_+(\fjrw{p'}{g'})
\end{align*}
and since $g^{-1}=(\epsilon/2, \alpha^{-1}, \epsilon/2,\beta^{-1})$, $(g')^{-1}=(\alpha^{-1},\beta^{-1})$ we have $\age g= \epsilon + \age (g')^{-1}$. Thus 
\begin{align*}
\deg^B_-(\fjrw{p}{g}) &= \deg p + \age g^{-1} -2\\ 
    &= (\tfrac{1-\epsilon}{2}+\tfrac{1-\epsilon}{2}+\deg p')+(\epsilon+\age (g')^{-1})-2 \\
    &= \deg p'+\age (g')^{-1}-1\\
    &= \deg^B_-(\fjrw{p'}{g'})
\end{align*}
\end{proof}

Now we have two isomorphisms $tw_A:\sA_{W,\s G}\to \sA_{\tw W,\tw G}$ and $tw_B:\sB_{W,\s G}\to \sB_{\tw W,\tw G}$ for A--models and 
B--models, resp. We will use these isomorphisms to prove Theorem~\ref{thm1}. 


\subsection{State Space isomorphism}\label{s:twistedstatespaceisom}
We now have all of the necessary tools to prove the main theorem, 
which we restate here:
\begin{customthm}{1}
There is an isomorphism of vector spaces (the mirror map) $\sA_{W,\s{G}}\cong \sB_{W^T,\s{G^T}}$, which preserves the bi--degree. 
\end{customthm} 

\begin{proof}
Recall the diagram 

\begin{center}\begin{tabular}{ccc}
$\sB_{\tw{(W^T)},\tw{(G^T)}}$&\begin{tikzpicture} \draw[<-] (0,0) -- (1,0); \node at (.5,.25) {$tw_B$}; \end{tikzpicture} &$\sB_{W^T, \s{G^T}}$ \\
\begin{tikzpicture} \draw[<->] (0,0) -- (0,1); \end{tikzpicture}& & \begin{tikzpicture} \draw[<->, dashed] (0,0) -- (0,1); \node at (0,1.1) {}; \end{tikzpicture} \\
$\sA_{\tw{W},\tw{G}}$ &\begin{tikzpicture} \draw[<-] (0,0) -- (1,0); \node at (.5,.25) {$tw_A$};\end{tikzpicture} &$\sA_{W, \s{G}}$ \\
\end{tabular}\end{center}

The two horizontal arrows are the content of Theorem~\ref{t:twist} and Corollary~\ref{c:Btwistiso}. The dashed arrow is the desired isomorphism of Theorem~\ref{thm1}. We will establish this map by providing the leftmost vertical arrow in the diagram. This will follow from Theorem~\ref{t:krawitz} (proven by Krawitz in \cite{Kr}). The only difference is the group of symmetries in the upper left corner is $\tw{(G^T)}$, instead of $\tw{(G)}^T$. Therefore all that remains is to show $\tw{(G^T)} = (\tw{G})^T$, which we do in the following lemma. 
\end{proof}

\begin{lem}
$\tw{(G^T)} = (\tw{G})^T$. 
\end{lem}
\begin{proof}
We first consider the inclusion $\tw{(G^T)}\subset (\tw{G})^T$. Recall $\tw{(G^T)}=(\phi^T)^{-1}(\s{G^T})$. Consider $g=(\alpha, \beta)\in \tw{(G^T)}$ and $g_\epsilon=((1/2)^{\epsilon},\alpha,(1/2)^{\epsilon},\beta)$ for $\epsilon\in\set{0,1}$. By definition, $g_0\in G^T$ or $g_1\in G^T$. 

We want to show that $g\in (\tw{G})^T$, which means we need to show that for each $h=(\gamma_x,\gamma_y)\in \tw{G}$, $gA_{\tw{W}}h\in \ZZ$. Again, $h\in \tw{G}$ means that $h_{\epsilon'}=((1/2)^{\epsilon'},\gamma_x,(1/2)^{\epsilon'},\gamma_y)\in G$ for either $\epsilon'=0$ or $\epsilon'=1$.

Let $A_1$ and $A_2$ be the exponent matrices for $f_1$ and $f_2$, resp. Then for $\epsilon = 0$ or 1, we know that 
\[
g_\epsilon A_{W}h_{\epsilon'}= 4(1/2)^{\epsilon+\epsilon'}+\alpha A_{1}\gamma_x+\beta A_{2}\gamma_y \in \ZZ.
\]
Since the first summand is an integer, we have $gA_{\tw{W}}h=\alpha A_{1}\gamma_x+\beta A_{2}\gamma_y \in \ZZ$. 

Now for the reverse inclusion $(\tw{G})^T\subset \tw{(G^T)}$. Suppose $g=(\alpha,\beta)\in (\tw{G})^T$, and again let $g_\epsilon= ((1/2)^{\epsilon},\alpha,(1/2)^{\epsilon},\beta)$. We need to show that either $g_0\in G^T$ or $g_1\in G^T$. In other words we need to show that $g_0A_{W}h\equiv 0\pmod \ZZ$ for every $h\in G$, or that $g_1A_{W}h\equiv 0\pmod \ZZ$ for every $h\in G$. For $h$ of the form $((1/2)^{\epsilon'},\gamma_x,(1/2)^{\epsilon'},\gamma_y)$, the previous calculation shows that this is indeed the case for both $g_0$ and $g_1$. The difficulty comes when $h$ is of the form $((1/2)^{\epsilon'},\gamma_x,(1/2)^{1-\epsilon'},\gamma_y)$. We consider two cases. 

First suppose $\age \alpha\equiv 0 \pmod \ZZ$. Then since $\tw{(G^T)}\subset (\tw{G})^T\subset \SL_{\tw{W}}$, $\age \beta\equiv 0 \pmod \ZZ$. We will show that, in fact, $g_0\in G^T$. Notice that $\alpha A_1\jw_x=\age \alpha\equiv0\pmod \ZZ$. Suppose first that, $h=((1/2),\gamma_x,0,\gamma_y) \in G$. Then $\jw_{1}+h\in \phi(\tw{G})\cap G$, and so 
\[
g_0 A_W (\jw_{1}+h)=\alpha A_1 \jw_x + \alpha A_1 \gamma_x + \beta A_2 \gamma_y = g A_{\tw{W}} \phi^{-1}(\jw_1 + h) \equiv 0\pmod \ZZ,
\]
since $\phi^{-1}(\jw_1 + h) \in \tw{G}$ and, by assumption,  $g \in (\tw{G})^T$.
Since also 
\[
g_0 A_W (\jw_{1}+h) = \age \alpha + g_0 A_W h,
\]
we have
$
g_0 A_W h \equiv 0\pmod \ZZ.
$
If $h=(0,\gamma_x,1/2,\gamma_y)$, the proof is the same, only using $\jw_{2}$ instead. 

Because $G\subset\SL_{W_1}\times\SL_{W_2}$, the only other case is $\age \alpha\equiv 1/2 \pmod\ZZ$. In this case $\age \beta\equiv 1/2 \pmod\ZZ$. We will show that, in fact, $g_1\in G^T$. The proof is similar, but this time, we notice that $\alpha A_1\jw_x$ is no longer an integer, but a half--integer. Again we first suppose that $h=((1/2),\gamma_x,0,\gamma_y)$. Then $\jw_{1}+h\in \phi(\tw{G})$, and so 
\[
g_1 A_W (\jw_{1}+h)=\alpha A_1 \jw_x + \alpha A_1 \gamma_x + \beta A_2 \gamma_y= g A_{\tw{W}} \phi^{-1}(\jw_1 + h) \equiv 0\pmod \ZZ
\]
since $\phi^{-1}(\jw_1 + h) \in \tw{G}$ and, by assumption,  $g \in (\tw{G})^T$. 
Since also 
\[
g_1 A_W (\jw_{1}+h) = \frac{1}{2}+\age \alpha + g_0 A_W h,
\]
we have $g_0 A_W h \equiv 0\pmod \ZZ$.

If $h=(0,\gamma_x,1/2,\gamma_y)$, the proof is the same, only using $\jw_{2}$ instead.
\end{proof}


This concludes the proof of Theorem~\ref{thm1}. 

\begin{rem}\label{rem:hodgeRotation}
	One may notice that the A--model bi-degree for $\sA_{W,\s G}$ and the B--model bi-degree for $\sB_{W,\s G}$ for the same BVLG model are related by what looks like the rotation of the Hodge diamond that is well-known for mirror symmetry of CY varieties. Thus if one views the bi--degree correctly, then the mirror map of Theorem~\ref{thm1} can similarly be viewed as a rotation of the Hodge diamond, as one expects in mirror symmetry.
\end{rem}

\section{A Geometric View}\label{s:geometry}

In this section, there are several aspects of the geometry of the Borcea--Voisin model we would like to consider. We begin in Section~\ref{ss:GLSM} by describing a construction called the GLSM, which justifies the choice of LG model $(W,\s{G})$, and gives some reason why we expect the LG/CY correspondence to hold in this case. In Section~\ref{ss:LGCYstatespace} we will describe the LG/CY state space isomorphism, establishing an equivalence between the BVLG model state space and the state space of the corresponding BV orbifold. In Section~\ref{ss:twistmap} we describe the twist map on the geometric side, generalizing the twist map of \cite{Borcea} and \cite{ABS}. 

\subsection{Gauged linear sigma models}\label{ss:GLSM}
We begin this section by describing why the LG model described in Section~\ref{ss:BVLGmodel} has the particular form we have described. None of the ideas presented in this section will be used in any later proofs. 

The LG/CY correspondence is part of a larger idea due to E. Witten, in which he considered each theory in the correspondence (e.g. GW theory or FJRW theory) as different ``phases'' of some larger theory called the \emph{gauged linear sigma model} (GLSM) that depends on some parameter. It is conjectured that variation of this parameter will then produce the various theories involved in the LG/CY correspondence. Thus different phases of the GLSM are conjectured to be equivalent to each other in some sense. The first evidence of this equivalence is an isomorphism of the respective state spaces in each phase. 

These ideas have recently been mathematically formalized by Fan--Jarvis--Ruan in \cite{FJR15} and we begin this section by briefly describing their construction. We are intentionally vague about some aspects of this construction, since they are not necessary for our purposes. In this article, we are primarily interested in the state spaces, so we will focus our attention there. 

A GLSM depends on a choice of 1) a finite dimensional vector space $V$ over $\CC$, 2) a reductive algebraic group $G_V$ acting on $V$, 3) a $G_V$--character $\theta$, and 4) a \emph{superpotential} $\overline W: V\to \CC$. From these ingredients, one obtains a \emph{state space}, a \emph{moduli space of LG--quasimaps} with a good virtual cycle, and numerical invariants defined as integrals over the virtual cycle.

The data of the vector space $V$, the group $G_V$, and the character $\theta$ yield a geometric object called a GIT quotient. If we vary the character $\theta$, we get a different GIT quotient, and therefore a different state space, moduli space, etc. We can vary the character in the so--called phase space. This space of characters is partitioned into various chambers and varying the character within a chamber does not change the theory. However, if we cross into a different chamber, we obtain a different theory. 

The idea behind the LG/CY correspondence is that if we choose the character $\theta$ in a certain chamber, we obtain GW theory of a particular orbifold, whereas if we choose the character in different chamber, we should expect to obtain FJRW theory. 
The moduli space, virtual cycle and numerical invariants from the GLSM will give rise to such structures for Gromov--Witten theory and for FJRW theory. It has been conjectured that these structures will somehow agree for both theories, but as mentioned we will focus our attention on the state spaces. 

The relevant GLSM's for BV models are obtained using the following input. Let $V=\CC^{n+m+2}\times \CC^2$ with coordinates $x_0,\dots,x_n,y_0,\dots,y_m$ and $p_1,p_2$. We define an action of $(\CC^*)^2$ on $V$ via the weights
\[
\left(\begin{matrix}
u_0 & \dots & u_n & 0 & \dots & 0 & -d_1 & 0\\
0 & \dots & 0 & v_0 & \dots & v_m & 0 & -d_2  
\end{matrix}
\right)
\]
We can embed $(\CC^*)^2$ as diagonal matrices into $\GL(V)$ via these weights.

We can similarly embed the group $\s{G}\hookrightarrow \GL(V)$ acting only on the $x$ and $y$ coordinates. Our reductive group $G_V$ we define as the subgroup of $\GL(V)$ generated by $(\CC^*)^2$ and $\s{G}$. 
For a superpotential, we take $\overline{W}=p_1W_1 + p_2W_2$. 

In order to describe the two relevant GIT quotients, we need to choose appropriate characters of $G_V$. In order to identify the characters, we use the method of symplectic reduction. For more on this perspective and the relationship between symplectic reduction and GIT quotients, see the original construction on GLSM by Fan--Jarvis--Ruan in \cite{FJR15}. We take the standard K\"ahler metric on $V$. Since $G_V$ is reductive, it is the complexification of a maximal compact Lie subgroup $H$ acting on $V$ via a faithful unitary representation. The Lie algebra in our case is $\RR^2$. 

We consider the Hamiltonian action of $H$ on $V$, which has moment map $\mu:V\to \RR^2$ given by 
\[
\mu_1 = \sum_{i=0}^n u_i|x_i|^2 - d_1|p_1|^2, \quad \mu_2 = \sum_{j=0}^m v_i|y_i|^2 - d_2|p_2|^2.
\]
The set of critical values for this moment map is $\set{\mu_1=0}\cup\set{\mu_2=0}\subset \RR^2$.

The conditions for the appropriate GIT quotients translate to the requirement of having a regular value of the moment map in the symplectic setting. Once we have our regular values we can return to the algebraic setting to describe the GIT quotients.  Notice the set of critical values divides $\RR^2$ into 4 chambers. Each regular value within a given chamber will yield an isomorphic GIT quotient. 

A derivation of a character defines a weight in the Lie algebra $\RR^2$. We will now discuss the two chambers which yield GW theory and FJRW theory. First we describe the regular values which yield the relevant characters; then we return the the GIT perspective, and describe the unstable locus and the corresponding GIT quotient for each character which has derivation in the given chamber.

\subsubsection*{GW phase} The first chamber we consider is defined by $\mu_1>0, \mu_2>0$. In this chamber, the unstable locus is the set of points in $\CC^{n+m+2}$ with $(x_0,\dots,x_n)\neq 0$, and $(y_0,\dots,y_m)\neq 0$. We get as a GIT quotient 
\[
\left[(\CC^{n+1}\setminus 0)\times (\CC^{m+1}\setminus 0)\times \CC^2/G_V\right]\cong \left[\cO_{\PP(u_0,\dots,u_n)}(-d_1)\times \cO_{\PP(v_0,\dots,v_n)}(-d_2)/\widetilde{\s{G}}\right]. 
\]
Here $\widetilde{\s{G}}=\s{G}/(J_1\times J_2)$. The critical locus of $\overline{W}$ is $\set{W_1=0,W_2=0, p_1=p_2=0}/\widetilde{\s{G}}$, which is the stack 
\[
\left[X_{W_1}\times X_{W_2}/\widetilde{\s{G}} \right].
\]
Though it has not been verified in every case, the state space corresponding to this GIT quotient is expected to be the Chen--Ruan cohomology of the critical locus of $\overline{W}$. In Section~\ref{ss:LGCYstatespace} , we will show that the Chen--Ruan cohomology of the Borcea--Voisin orbifold is indeed isomorphic to the state space of FJRW theory. If we set $n=2, m=3$ this has a Borcea--Voisin variety as crepant resolution.

\subsubsection*{FJRW phase} The other relevant chamber is defined by $\mu_1<0, \mu_2<0$. The unstable locus is the set of points in $V$ with $p_1\neq 0$, and $p_2\neq 0$. We get the GIT quotient 
\[
\left[\CC^{n+m+2}\times (\CC^*)^2/G_V\right]\cong [\CC^{n+m+2}/\s{G}] 
\]
with superpotential $\overline{W}=W$ (we scaled away the $p$'s).

The critical locus of $\overline{W}$ is the origin. In this chamber, the state space is exactly the state space of FJRW theory.   

Because both theories are merely different phases of the same GLSM, we expect the theories to be equivalent. The most basic manifestation of this equivalence is an isomorphism of state spaces, which we show in Section~\ref{ss:LGCYstatespace}.

\subsection{LG/CY correspondence: state space isomorphism}\label{ss:LGCYstatespace}

In this section, we will prove that the two state spaces described above are isomorphic as graded vector spaces, i.e. 
\begin{equation}\label{e:stateiso}
\sA_{W,\s{G}}^{p,q}\cong H_{CR}^{p,q}\Big(\big[X_{W_1}\times X_{W_2}/\widetilde{\s{G}} \big];\CC\Big) 
\end{equation}
This is similar to what was done by Chiodo--Ruan in \cite{ChR}. However, there are some important differences, so we will give the details here. 
We begin by expressing each side of Equation~\eqref{e:stateiso} in a more useable form.

\subsubsection{FJRW state space} 
We begin with FJRW theory. 
In order to simplify notation, we will write $J=J_1\times J_2$. For $g\in \s{G}$, we write $g=(g_1,g_2)$ with $g_1\in G_{W_1}^{max}$ and $g_2\in G_{W_2}^{max}$. Notice this differs slightly from previous sections, where we wrote $g=(\epsilon/2,\alpha,\epsilon'/2,\beta)$. 

Finally, define the following notation:
\begin{align*}
\CC^{n+1}_{g_1}&=\set{x\in \CC^{n+1}\mid g_1x=x}\\
W_{1,g_1}&=W|_{\CC^{n+1}_{g_1}}\\
\sQ_{g_1}&=\sQ_{W_1,g_1}\omega_{g_1}
\end{align*}
with similar definitions for $g_2$. 

Recall the definition of the state space:
\begin{equation}\label{e:fjrwstatespacereview}
\sA^{p,q}_{W,\s{G}}=\bigoplus_{g\in\s{G}}(\sQ_{W_g}^{p,q}\omega_g)^{\s{G}}
\end{equation}

Since $W_g=W_{1,g_1}+W_{2,g_2}$, we obtain $\sQ_{W_g}=(\sQ_{W_1,g_1}\otimes \sQ_{W_2,g_2})$. 
Thus in the $g$-sector, we can write 
\begin{equation}\label{e:fjrwtensor}
\sQ_{W_g}^{p,q}\omega_g=\bigoplus_{\substack{h_1+h_2=p\\k_1+k_2=q}}(\sQ^{h_1,k_1}_{g_1}\otimes \sQ^{h_2,k_2}_{g_2})^{\s{G}}.
\end{equation}

This isomorphism sends
\[
\fjrw{m}{g}\to \fjrw{m_1}{g_1}\otimes\fjrw{m_2}{g_2}
\]
where $m=m_1m_2$. Recall the definition of bi-degree for FJRW theory in Equation~\eqref{e:fjrwbidegree}. Notice $\age g=\age g_1+\age g_2$ and $N_g=N_{g_1}+N_{g_2}$. Therefore the tensor product preserves the bigrading, i.e. the bi-degree on the left hand side of \eqref{e:fjrwtensor} is equal to the sum of the bi-degrees on the right hand side. 

Since $G\subset\SL_{W_1}\times \SL_{W_2}$, $\sigma\notin G$. We can decompose $\s{G}$ into $2M=2|G|/(d_1d_2)$ cosets of $J$. We can choose one representative for each coset so that the first $M$ coset representatives fix $x_0$ and $y_0$, and the last $M$ coset representatives are simply the first ones multiplied by $\sigma$.  We denote by $\cC$ the set of these representatives.  

Now we can write the degree $(p,q)$ part of the state space as a sum over these cosets. Using \eqref{e:fjrwstatespacereview} and \eqref{e:fjrwtensor} this becomes
\begin{equation}\label{e:FJRWdecomposition}
\sA_{W,\s{G}}^{p,q}=\bigoplus_{g\in \cC}\bigoplus_{\substack{h_1+h_2=p\\k_1+k_2=q}}\left[\left(\bigoplus_{k=1}^{d_1}\sQ_{g_1\jw_1^k}^{h_1,k_1}\right)^{J_1}\otimes\left(\bigoplus_{k=1}^{d_2}\sQ_{g_2\jw_2^k}^{h_2,k_2}\right)^{J_2}\right]^{\s{G}}.    
\end{equation}

Recall that $\jw_i$ is the generator of $J_i$ for $i=1,2$. Notice we have taken $J_1$ and $J_2$ invariants in each of the factors of the tensor product in the second line. We do this simply to make the isomorphism more clear.

\subsubsection{Chen--Ruan cohomology}
The Chen--Ruan cohomology is slightly more subtle. On the Calabi--Yau side, the state space is 
\[
H^*_{CR}\Big(\left[X_{W_1}\times X_{W_2}/\widetilde{\s{G}}\right]\Big).
\]

The \emph{Chen--Ruan orbifold cohomology} is defined via the \emph{inertia orbifold} (see \cite{ChR}).
If $\cX = [X/H]$ is a global quotient of a nonsingular
variety $X$ by a finite group $H$, the inertia orbifold $I\cX$ takes a particularly simple form.
Let $S_H$ denote the set of conjugacy classes $(h)$ in $H$,
then
\[
I\cX = \coprod_{(h) \in S_H} [ X^h/C(h) ].
\]
As a vector space, the Chen--Ruan cohomology groups $H^*_{CR}(\cX)$ 
of an orbifold $\cX$ are the cohomology groups of 
its inertia orbifold:
\[
H_{CR}^*(\cX) := H^*(I\cX).
\]
The bi-degree on the Chen--Ruan cohomology is the normal bi-degree with an \emph{age shift}, which we will describe below. 

In order to compute the Chen--Ruan cohomology, we need to describe the orbifold $\left[X_{W_1}\times X_{W_2}/\widetilde{\s{G}}\right]$ in a different form.
We write $T^2$ for the torus $(\CC^*)^2$, acting on $\CC^{m+n+2}$ via the weights 
\[
\left(\begin{matrix}
u_0 & \dots & u_n & 0 & \dots & 0 \\
0 & \dots & 0 & v_0 & \dots & v_m   
\end{matrix}
\right)
\]
and the group $\s{G}T^2$ for the product of the two groups. Notice that $\s{G}\cap T^2=J$. We will denote by $\mu_{d_1}$ the cyclic group of order $d_1$ in the first copy of $\CC^*$ and $\mu_{d_2}$ the corresponding one in the second factor. Notice that $\mu_{d_i}$ corresponds with the cyclic group $J_i$. 

Define 
\begin{align*}
V_{W_1}&=\set{W_1=0}\subset \CC^{n+1}\setminus \set{0}    \\
V_{W_2}&=\set{W_2=0}\subset \CC^{m+1}\setminus \set{0}.
\end{align*}
Notice that, since $\s{G}\cap T^2=J$, we have $\s{G}T^2/T^2\cong \widetilde{\s{G}}$. Using this description, we can express the orbifold (see \cite{romag}) as 
\[
\left[X_{W_1}\times X_{W_2}/\widetilde{\s{G}}\right]\cong \left[V_{W_1}\times V_{W_2}/\s{G}T^2\right].
\]
 
The Chen--Ruan cohomolgy can therefore be written as 
\begin{equation}\label{e:CRdirectsum}
H^{p,q}_{CR}\Big(\left[X_{W_1}\times X_{W_2}/\widetilde{\s{G}}\right]\Big)=\bigoplus_{g\in\s{G}T^2}H^{p-a(g),q-a(g)}((V_{W_1}\times V_{W_2})_{g}/\s{G}T^2;\CC).
\end{equation}
Here $a(g)$ denotes the age shift mentioned earlier. To define this we look at the action of $g$ on the tangent space at a given point. The action can be diagonalized in a suitable basis as
\[
\left(\begin{matrix}
e^{2\pi i a_1} & & 0 \\
& \ddots & \\
0& & e^{2\pi i a_{n+m}}
\end{matrix}\right).
\]
with $0\leq a_i<1$. The age $a(g)=\sum_{l=1}^{m+n}a_l$. Notice this differs from the age in FJRW theory, because of $g$ acting on the tangent space to $V_{W_1}\times V_{W_2}$ instead of the action on $\CC^{n+m+2}$. The difference in the two age contributions is discussed at some length in \cite{ChR}.

With this description of the Chen--Ruan cohomology,we can now write the Chen--Ruan cohomology of the Borcea--Voisin orbifold in a more suitable manner. Given $g\in \s{G}T^2$, we can write $g=(g_1,g_2)$ with $g_1$ acting only on the $x$'s and $g_2$ only on the $y$'s, as we did with FJRW theory. We can define $\CC^{n+1}_{g_1}$ and $W_{1,g_1}$ similar to what was in the previous section, and we define the hypersurfaces
\begin{align*}
V_{W_1,g_1}&=\set{W_{1,g_1}=0}\subset\CC^{n+1}_{g_1}\setminus \set{0}\\
V_{W_2,g_2}&=\set{W_{2,g_2}=0}\subset\CC^{m+1}_{g_2}\setminus \set{0}.
\end{align*}

Now consider a single summand corresponding to $g\in\s{G}T^2$. We can write $(V_{W_1}\times V_{W_2})_g=V_{W_1,g_1}\times V_{W_2,g_2}$
\[
(V_{W_1}\times V_{W_2})_g/\s{G}T^2\cong (V_{W_1,g_1}\times V_{W_2,g_2}/T^2)/\widetilde{\s{G}}.
\]
But 
\[
V_{W_1,g_1}\times V_{W_2,g_2}/T^2\cong X_{W_1,g_1}\times X_{W_2,g_2}.
\]

So for a fixed $g\in \s{G}T^2$, we can write the cohomology as 
\begin{align*}
H^{p-a(g),q-a(g)}&(V_{W_1,g_1}\times V_{W_2,g_2}/\s{G}T^2;\CC) = H^{p-a(g),q-a(g)}(X_{W_1,g_1}\times X_{W_2,g_2}/\widetilde{\s{G}};\CC)\\
	&\qquad\qquad\qquad\qquad\qquad\qquad=H^{p-a(g),q-a(g)}(X_{W_1,g_1}\times X_{W_2,g_2};\CC)^{\widetilde{\s{G}}}\\
	&=\bigoplus_{\substack{h_1+h_2=p\\ k_1+h_2=q}}\left(H^{h_1-a(g_1),k_1-a(g_1)}(X_{W_1,g_1};\CC)\otimes H^{h_2-a(g_2),k_2-a(g_2)}(X_{W_2,g_2};\CC)\right)^{\widetilde{\s{G}}}.
\end{align*}

Notice that the summand corresponding to $g$ in Equation~\eqref{e:CRdirectsum} is zero if $(V_{W_1}\times V_{W_2})_{g}$ is empty. As with FJRW theory, we will rewrite this sum as a sum over cosets. To this end, if we write $g_1=(g_{10},g_{11}, \dots, g_{1n})\in G_{W_1}$ and $g_2=(g_{20}, \dots, g_{2m})$ we can define 
\[
\Lambda^1_{g}=\bigcup_{k=0}^n\set{\lambda\in \CC^*\mid\lambda^{-w_k}=g_{1k}}.
\] 
This set denotes the complex numbers $\lambda\in \CC^*$ where $g_1\lambda$ has non-trivial fixed locus. The set $\Lambda^2_{g}$ is defined similarly. 

Notice that in this description, since the tangent space of a product is the product of the respective tangent spaces, we have $a(g)=a(g_1)+a(g_2)$. So the degree shift agrees on both sides of the equation. In other words, we can think of the summands as tensor product of factors of some Chen--Ruan cohomology. 

As in the FJRW state space, we can choose $2M$ cosets of $T^2$, using the same set $\cC$ of coset representatives as before. We can then write the Chen--Ruan cohomology as 
\begin{align}
H_{CR}^{p,q}&([X_{W_1}\times X_{W_2}/\s{G}];\CC)=\nonumber\\
	&\bigoplus_{g\in \cC}\bigoplus_{\substack{h_1+h_2=p\\k_1+k_2=q}}\left[\left(\bigoplus_{\lambda\in \Lambda^1_{g}}H^{h_1,k_1}(X_{W_1,g\lambda};\CC)\right)\otimes\left(\bigoplus_{\lambda\in \Lambda^2_{g}}H^{h_2,k_2}(X_{W_2,g\lambda};\CC)\right)\right]^{\s{G}}.\label{e:CRdecomposition}
\end{align} 
On both sides of this equation, we have written bi-degrees \emph{with} the age shift. This is standard on the left hand side of the equation. On the right hand side, however, we mean the $h_1-a(g_1),k_1-a(g_2)$ part of the ordinary cohomology.

\subsubsection{The Isomorphism}

Comparing expressions \eqref{e:FJRWdecomposition} and \eqref{e:CRdecomposition}, we see that the isomorphism of state spaces will follow, once we establish the isomorphism
\begin{equation}\label{e:sectoriso}
\left(\bigoplus_{k=1}^{d_i}\sQ_{g_i\jw_i^k}^{h_i,k_i}\right)^{J_i}
\cong \bigoplus_{\lambda\in \Lambda^i_{g}}H^{h_i,k_i}(X_{W_i,g_i\lambda};\CC).
\end{equation}
as an isomorphism of bigraded vector spaces. Again here, we take the convention that the right hand side of \eqref{e:sectoriso} is shifted by the age shift and the left hand side has the bi-degree defined by FJRW theory. We are summing over $J_1$ and the elements of $\CC^*$ contributing a nonzero sector to the Chen--Ruan cohomology, resp.   

Furthermore, because this isomorphism is $G_{W_i}^{max}$-equivariant, the action of $G_{W_i}^{max}$ is the same on both sides of the isomorphism (see \cite{ChR}), and hence the action of $\s{G}$ is the same on the summands of both \eqref{e:FJRWdecomposition} and \eqref{e:CRdecomposition}. 

Equation \eqref{e:sectoriso} is proven by Chiodo--Ruan in \cite{ChR}, but we give a brief outline here. For ease of exposition, we will focus on the isomorphism with $i=1$. The same will be true for $i=2$. On the left hand side of \eqref{e:stateiso} we have 
\begin{equation}\label{e:FJRWsummand}
\left(\bigoplus_{k=1}^{d_1}\sQ_{g_1\jw_1^k}^{h_1,k_1}\right)^{J_1}=\bigoplus_{\lambda\in\mu_{d_1}\cap\Lambda^1_{g}}H^{N_{g_1\lambda}}(\CC^{n+1}_{g_1\lambda}, W^{+\infty}_{g_1\lambda};\CC)^{J_1}\oplus\bigoplus_{\lambda\in\mu_d\setminus\Lambda^1_g}1_{g_1\lambda}\CC.
\end{equation}
This is the definition of the various sectors of FJRW theory. Notice we have decomposed the left hand side into a sum of broad sectors and narrow sectors. The action of $G_{W_1}^{max}$ is trivial on narrow sectors. 

On the other hand, we can express the cohomology of a hypersurface in weighted projective space as a direct sum of the ambient cohomology (coming from projective space) and the primitive cohomology. In e.g. \cite{dimca, dolg, ChR} we see that the primitive cohomology can be expressed as the cohomology of the Milnor fiber invariant under the monodromy action. Thus we can write:
\[
H^{*}(X_{W_1,g_1};\CC)\cong H^{N_{g_1\lambda}}(\CC^{n+1}_{g_1\lambda}, W^{+\infty}_{g_1\lambda};\CC)^{J_1}\oplus H^{amb}(X_{W_1,g_1})
\]
In this description, $G_{W_1}^{max}$ acts trivially on the ambient classes, and the group action on the primitive cohomology is the same as the action on the FJRW state space. 

One last thing to note is that the Chen Ruan cohomology contains a sum over elements $\lambda\in \CC^*$, or rather a sum over $\lambda\in \Lambda^1_g$, since only these contribute to the cohomology. It may happen that for some of these the corresponding diagonal symmetry does not lie in $J_1$, i.e. when $\lambda\notin\mu_{d_1}$ in the notation of the decompositions written above. In this case, $W_{1,g_1\lambda}$ vanishes on all of $\CC^{n+1}_{g_1\lambda}$ (see \cite{ChR}), and therefore the primitive cohomology vanishes. For these summands, we obtain simply the cohomology of weighted projectives space (the ambient classes). 

Thus we have 
\begin{equation}\label{e:CRsummand}
\bigoplus_{\lambda\in \CC^*}H^{\bullet}(X_{W_1,g_1\lambda};\CC)\cong\bigoplus_{\lambda\in\mu_{d_1}\cap\Lambda^1_{g}} H^{N_{g_1\lambda}}(\CC^{n+1}_{g_1\lambda}, W^{+\infty}_{g_1\lambda};\CC)^{J_1}\oplus \bigoplus_{\lambda\in\Lambda^1_{g}}H^{amb}(X_{W_1,g}).
\end{equation}
In this expression, $G_{W_1}^{max}$ acts trivially on the ambient classes. 

Comparing expressions \eqref{e:FJRWsummand} and \eqref{e:CRsummand}, we see that we only need to compare the narrow sectors and the ambient classes. The degree shift for the broad sectors and for the primitive classes in the Chen--Ruan cohomology agree because the two age shifts agree (see \cite[Lemma 22]{ChR}). 

The key observation in proving \eqref{e:sectoriso} is that the number of narrow sectors in FJRW theory is equal to the sum of the dimension of all of the primitive cohomology, when one sums over $\lambda\in \CC^*$. Furthermore, the bi-degrees of the classes also agree, after the degree shift on both sides. The interested reader can read the details in \cite{ChR}. As mentioned before, $G_{W_1}^{max}$ acts trivially on both the narrow sectors of FJRW theory and the ambient classes of the Chen--Ruan cohomology. 

This establishes \eqref{e:sectoriso} for $i=1$. We can do the same with $W_2$ obtaining a similar isomorphism. Putting these together with \eqref{e:FJRWdecomposition} and \eqref{e:CRdecomposition}, we obtain the isomorphism of state spaces predicted by the LG/CY correspondence.

\subsection{Borcea--Voisin mirror symmetry}
One particular consequence of our main theorem is a geometric statement regarding the cohomology of the corresponding BV orbifolds. Indeed, one of the first predictions of mirror symmetry is the rotation of the Hodge diamond. Such a statement is one of the main results of Chiodo--Kalashnikov--Veniani in \cite{CKV} for the Borcea--Voisin orbifolds similar to the form we consider here. In other words, they prove the following corollary in case $G=G_1\times G_2$ with $G_i\in G_{W_i}^{max}$. Here we can drop that condition. 

\begin{cor}\label{cor:BVMS} With $N=n+m-2$
	\[
	H^{p,q}_{CR}\big(\big[X_{W_1}\times X_{W_2}/\widetilde{\s{G}}\big];\CC\big)\cong H^{N-p,q}_{CR}\big(\big[X_{W_1^T}\times X_{W_2^T}/\widetilde{\s{G^T}}\big];\CC\big)
	\]
\end{cor}

\begin{proof}
	From Section \ref{s:twistedstatespaceisom} we have 
	\[
	\sA_{W,\s{G}}^{p,q}\cong \sB_{W^T,\s{G^T}}^{p,q}.
	\]
	
	On the other hand, from Section \ref{ss:LGCYstatespace}, we have 
	\[
	\sA_{W,\s{G}}^{p,q}\cong H^{p,q}_{CR}(X_{W_1}\times X_{W_2}/\widetilde{\s{G}};\CC). 
	\]
	
In fact, based on Remark~\ref{r:ABiso} and Remark~\ref{rem:hodgeRotation}) we see that 
\[
\sB_{W^T,\s{G^T}}^{p,q}\cong H^{N-p,q}_{CR}(X_{W_1^T}\times X_{W_2^T}/\widetilde{\s{G^T}};\CC). 
\]
\end{proof}


If we relax the condition that $G$ be a product, we no longer are considering Borcea--Voisin orbifolds, as we have seen previously, but from the LG side, we expected mirror symmetry to hold nonetheless.

In case $n=2$, $m=3$ and $G=G_1\times G_2$ is a product of groups $G_1\subset G_{W_1}^{max}$, $G_2\subset G_{W_2}^{max}$, then a crepant resolution of both orbifolds exists, and we obtain the mirror symmetry of Borcea--Voisin at the level of state spaces.

\subsection{The twist map}\label{ss:twistmap}
Now we turn our attention to the twist map of Borcea in \cite{Borcea} and Artebani--Boissi\`ere--Sarti in \cite{ABS}. Geometrically, the twist map relates the orbifold $[X_{W_1}\times X_{W_2}/\widetilde{\s{G}}]$ to a hypersurface in a quotient of weighted projective space. What follows is a particulary nice instance where the LG side of the LG/CY correspondence informs the CY geometry, allowing us to generalize Borcea's twist map. 

In \cite{ABS} the authors describe the twist map for those pairs of polynomials $W_1$, $W_2$ with the property that $\gcd(u_0,v_0)=1$. As mentioned before, there are 44 of the 95 weight systems for K3 surfaces that admit such a polynomial, and two weight systems for an elliptic curve. Thus we have 88 total combinations. The restriction $\gcd(u_0,v_0)=1$ limits us to 48 of these combinations. However, as mentioned in Section \ref{sec:twist}, the LG/CY correspondence allows us to understand the twist map more clearly. For example, the restriction on gcd's can be lifted so that the twist map is valid for all choices of polynomials that have form \eqref{e:wxwy}, as soon as we understand the group $\tw{G}$. 

In order to define the twist map, let $\delta=\gcd(u_0,v_0)$. We define $s_0$ and $t_0$ via the equations $s_0u_0+\delta\equiv 0\pmod {v_0}$ and $t_0v_0+\delta\equiv 0\pmod {u_0}$. Let $\sigma$ be the involution on $\mathbb{P}(u_1, \ldots, u_m) \times \mathbb{P}(v_0, v_1, \ldots, v_n)$ given by $(x_0, y_0) \mapsto (-x_0, -y_0),$ 
Finally, define 
\[
s=\frac{s_0u_0+\delta}{v_0},\quad t=\frac{t_0v_0+\delta}{u_0}.
\]

Recall in the notation of Equation~\eqref{e:Wform} that the LG twist map relates the LG model $(W_1+W_2,\s{G})$ to the LG model $(f_1-f_2,\tw{G})$. One can check that $\tw{(J_1\times J_2)}\subset \tw{G}$ and contains $J_{f_1-f_2}$ and that $\tw{(J_1\times J_2)}/J_{f_1-f_2}$ is cyclic of order $\delta$. Thus this group acts on $\mathbb{P}(\frac{v_0}{\delta}u_1, \ldots, \frac{v_0}{\delta} u_m, \frac{u_0}{\delta}v_1, \ldots, \frac{u_0}{\delta}v_n)$ (see Section~\ref{sec:twist}). We define the map

\begin{align*}
\tilde{\tau}: & \mathbb{P}(u_0, u_1, \ldots, u_m) \times \mathbb{P}(v_0, v_1, \ldots, v_n) 
\ \to \ \left[\mathbb{P}\left(\frac{v_0}{\delta}u_1, \ldots, \frac{v_0}{\delta} u_m, \frac{u_0}{\delta}v_1, \ldots, \frac{u_0}{\delta}v_n\right)/\widetilde{\tw{G}}\right]
\end{align*} 
by 
\begin{align*}
((x_0, x_1, \ldots, x_n), (y_0, y_1, \ldots, y_m)) \mapsto \left(\left(x_0^{s_0}y_0^t\right)^{\frac{u_1}{\delta}}x_1, \ldots, \left(x_0^{s_0}y_0^t\right)^{\frac{u_m}{\delta}}x_m, \left(x_0^sy_0^{t_0}\right)^{\frac{v_1}{\delta}}y_1, \ldots, \left(x_0^sy_0^{t_0}\right)^{\frac{v_n}{\delta}}y_m\right).
\end{align*} 

This map depends on a choice of $\delta$-th root of unity, and on the choice of $s_0$ and $t_0$. However, one can check that with a different choice of any of these, the image differs exactly by the action of an element of $\tw{(J_1\times J_2)}$. Since the image lands in the quotient by $\tw{G}$, the map is well-defined.

Furthermore, the image written above is equivalent to 
\[
\left(\left(\frac{y_0}{x_0}\right)^{\frac{u_1}{u_0}}x_1, \ldots, \left(\frac{y_0}{x_0}\right)^{\frac{u_m}{u_0}}x_m, y_1, \ldots, y_n\right).
\]
from which we see that $\tau$ descends to a well-defined map on the orbits of $\sigma$. 

The twist map is defined as the restriction of $\tau$ to the product $X_{W_1}\times X_{W_2}$. We need to check the image of $\tau$ is contained in  $\{f_1-f_2=0\}$. Recall that the weights of $x_0$ and $y_0$ are $1/2$, so we have $2u_0=d_1$ (the degree of the first polynomial), and $2v_0=d_2$. From the definition of $s$, we obtain
\[
\frac s\delta d_2=\frac{s_0}{\delta}d_1+2 
\]
\[
\frac t\delta=\frac{t_0}{\delta}+2. 
\]
Evaluating $f_1-f_2$ at any point in the image of $\tau$, we get 
\[
x_0^{s_0d_1/\delta}y_0^{t_0d_2/\delta}(y_0^{2}f_1(x)-x_0^2f_2(y))=0 
\]

And we obtain the map 
\[
\tau:[X_{W_1}\times X_{W_2}/\widetilde{\s{G}}]\to [X_{f_1-f_2}/\widetilde{\tw{G}}]
\]

For a given choice of root, this map is smooth and a diffeomorphism almost everywhere. Furthermore, one can see from the definitions that if $\delta=1$, we obtain the twist map of \cite{ABS}, which is in fact a birational morphism.  

So in the case that both the domain and image are Calabi-Yau threefolds, and thus related by a sequence of simple flops, the two orbifolds related by the twist map have equivalent genus-zero Gromov-Witten invariants (as shown in \cite{LR}). In general, it is more difficult to describe the relationship between Gromov--Witten invariants, or between the Chen--Ruan cohomologies (see e.g. \cite{BG, LLW} for a more detailed account). However, using the LG/CY correspondence, we will see that the Chen--Ruan cohomology of the two objects related by the twist map are indeed isomorphic, even in higher dimensions, as in the following corollary. This relates the Chen--Ruan cohomology of a product with the cohomology of a hypersurface.



\begin{cor}There is an isomorphism of bigraded vector spaces
	\[
	H^{p,q}_{CR}\Big(\big[X_{W_1}\times X_{W_2}/\widetilde{\s{G}}\big];\CC\Big)\cong H^{p,q}_{CR}\Big(\big[X_{f_1-f_2}/\widetilde{\tw{G}}\big];\CC\Big)
	\]
\end{cor}

\begin{proof}
	From Section \ref{ss:LGCYstatespace} we have
	\[
	\sA^{p,q}_{W,\s{G}}\cong H^{p,q}_{CR}\Big(\big[X_{W_1}\times X_{W_2}/\widetilde{\s{G}}\big];\CC\Big).
	\]
	From \cite{ChR} we have
	\[
	\sA^{p,q}_{\tw{W},\tw{G}}\cong H^{p,q}_{CR}\Big(\big[X_{f_1-f_2}/\widetilde{\tw{G}}\big];\CC\Big).
	\]
	Finally from Section \ref{s:twistedstatespaceisom}, we have
	\[
	\sA^{p,q}_{W,\s{G}}\cong\sA^{p,q}_{\tw{W},\tw{G}}
	\]
\end{proof}

\section{Landau--Ginzburg Algebra Isomorphism}\label{sec:alg_isom}

Finally, we return to the question mentioned in the introduction about Frobenius algebra structure. Recall that a Frobenius algebra is a $\ZZ_2$-graded supercommutative $\CC$-algebra together with a pairing that satisfies the Frobenius property
\[
\langle\alpha,\beta\cdot\gamma\rangle=\langle\alpha\cdot\beta,\gamma\rangle.
\]
We have shown mirror symmetry holds on the level of bigraded vector spaces. But as we have remarked, each state space also has the structure of a Frobenius algebra over $\CC$. It is expected that we should also find an isomorphism of Frobenius algebras. 

There are several reasons for including this in the final section. The first reason is that it is not the main result of the paper, but we find it interesting, nonetheless. Our method exhibits an interesting relationship between different mirrors of the same object (see Remark~\ref{rem:multmir}). 

The second reason has to do with the issues mentioned in the introduction. The Frobenius algebra structure for the A--model was defined in \cite{FJR13}, however, to our knowledge, nobody has yet considered the supercommutativity of the Frobenius algebra. In fact, most thought that the A--model state space was commutative. Part of the reason for this, is that most papers only deal with the the $G^{\max}$-invariant subspace of the full state space. In this article we expand our view slightly, and consider the even-graded subspace.  

And finally, on the B--side, there is some question as to the proper definition of the product for the Frobenius algebra. There are several suggested definitions, including that of Krawitz \cite{Kr}, Basalaev--Takahashi--Werner \cite{BTW}, and He--Li--Li \cite{HLL}. Again the issue of supercommutativity arises, since the definition of Krawitz produces a commutative Frobenius algebra, instead of a supercommutative one. However, if we restrict to the even graded subspace, the supercommutativity becomes commutativity. 

In this article, we will exploit the work of Francis--Jarvis--Johnson--Suggs in \cite{FJJS}, who use Krawitz' definition of the B--model Frobenius algebra. This is the only work we know of that really deals with the Frobenius algebra structure for LG mirror symmetry. 

In order to avoid the issues just mentioned, we restrict ourselves to the even-graded part of each state space. It is perhaps important to note that the $G^{max}$ invariant subspace in the $A$--model state space is even-graded, (see \cite[Lemma 1.4]{Kr}). This is a particularly important subspace of the state space (see e.g. \cite[Conjecture 4.8]{ChRglobal}) as most instances of the LG/CY correspondence are proven only for this subspace. 

We begin with the definitions of the pairing and algebra structure for both the A-- and B--models. Then we will prove that they are isomorphic as Frobenius algebras.

\subsection{A--model Frobenius algebra}
For the A--model we will first define the pairing, but we will not define the product fully, instead relying on the B--model to prove equivalence. 
Recall that the A--model state space is given by $\sA_{W,\s{G}}$ as in Section~\ref{ss:LGA}. For $\fjrw{m}{g}\in \sA_{W, \s{G}}$, the $\ZZ_2$ grading is defined simply by $N_g\pmod 2$. Notice, in our case, this agrees with the A--model grading, since 
\[
\deg_+^A(\fjrw{m}{g})+\deg_-^B(\fjrw{m}{g})=N_g+2\age g-2\age \jw. 
\]

As mentioned, we restrict our attention to the even-graded subspace of $\sA_{W, \s{G}}$, which we denote by $\sA^0$.

Now we turn our attention to the pairing. Recall the definition of $\sA_{W, \s{G}}$ involved some spaces of Lefschetz thimbles, which are equipped with a natural non--degenerate pairing:
\[
H^{mid} (\Fix(g),(W)_g^{-1}( \infty))^G \times H^{mid} (\Fix(g),(W)_g^{-1}( \infty))^G\to \CC. 
\]

The Milnor ring $\sQ_W$ also has a natural residue pairing $\langle \cdot , \cdot\rangle_W$ determined by the equation 
\begin{equation}\label{e:milpairing}
f\cdot g =  \frac{\Hess(W)}{\mu_W}\langle f, g\rangle_W+ \text{ lower order terms}.
\end{equation}

In our case the identification between the two descriptions of the A--model in Theorem~\ref{t:Wall} also respects the pairing on both spaces. We can now describe the A--model pairing on $\sA^0$. 
\begin{equation}\label{e:pairing}
\left\langle\fjrw{m}{g},\fjrw{n}{h} \right\rangle_A = \left\{ \begin{array}{cc}
\langle m,n\rangle_{W_g} & if \quad h = g^{-1}\\
0		& \text{otherwise}.
\end{array}\right.
\end{equation}
Here the pairing on the right--hand side is the pairing from the Milnor ring. Notice that this pairing is well--defined, since $W_g = W_{g^{-1}}$, and so $(\sQ_{g}\omega_g)^{G} \cong (\sQ_{g^{-1}}\omega_{g^{-1}})^{G}$. This agrees with the original pairing on the state space described above. Notice that the pairing only pairs even classes with even classes (with respect to the $\ZZ_2$ grading).   
We will use $\eta$ to denote the matrix associated to this pairing in the basis described above. 

\begin{rem}
	The pairing can be defined on the entire space, but we will avoid this technicality.
\end{rem}

\begin{rem}\label{pairing:conditions}
	The above implies that  $\langle\fjrw{m}{g},\fjrw{n}{h} \rangle \neq 0$ exactly when $g = h^{-1}$ and one of the monomials of $m\cdot n$ is a scalar multiple of $\Hess(W_g)$. By $m\cdot n$ we mean the product of the monomials multiplied by the appropriate volume form $\omega_g$. 
\end{rem}

The product on the A--model is defined in \cite{FJR13} via the structure constants
\[
a\star_{\sA} b:= \sum_{\xi,\xi'}\langle a,b,\xi \rangle_{0,3} \eta^{\xi,\xi'}\xi'
\]
where the sum runs over a basis of $\sA_{W,G}$, and $\eta^{\xi,\xi'}$ are the corresponding entries from the matrix inverse to the pairing matrix $\eta$. 

The structure constants $\langle a,b,\xi \rangle_{0,3}$ are defined in FJRW theory via certain integrals over the moduli space of curves $\overline \cM_{0,3}$, and there are corresponding numbers for higher genus and more marked points as well (see \cite{FJR13}).  Explicit computations of some of these constants in certain cases are given in \cite{FJR13, kpabr, D4, Kr, FJJS} and most recently in \cite{HLSW}, and methods for computations are given in \cite{Guere, Francis}. 

One fact regarding these structure constants that we should note in this case: if two of the entries are even-graded, the structure constant vanishes, unless the third is also even-graded (see \cite{FJR13}). Thus the even-graded subspace forms a subalgebra.

We will not need these numbers explicitly in this work. We rely instead on the B--model.

\subsection{B--model Frobenius algebra}\label{ss:BmodelAlg}
For the B--model we will also first define the pairing, and then we will define the product. 

Recall that the B--model state space is given by $\sB_{W^T,\s{G^T}}$. For $\fjrw{m}{g}\in \sB_{W^T, \s{G^T}}$, the $\ZZ_2$ grading is defined simply by $N-N_g\pmod 2$. Again, in this case, this agrees with the B--model grading, since 
\begin{align*}
\deg_+^B(\fjrw{m}{g})+\deg_-^B(\fjrw{m}{g})&=\age g+\age g^{-1}+2\deg m-2\age \jw\\
	&=N-N_g+2\deg m-2\age \jw.
\end{align*}

As mentioned, we restrict our attention to the even-graded subspace of $\sB_{W^T, \s{G^T}}$, and again denote it by $\sB^0$.

For the B--model, the pairing is defined as in \eqref{e:milpairing} and \eqref{e:pairing}. This differs from the pairing defined in \cite{HLL} only in the factor $\mu_g$ and from the pairing in \cite{BTW} by the same factor and an additional overall factor of $-|\s{G}|$. 

On the B--side, the product was defined by Intriligator, Vafa, and Kaufmann \cite{IV, kau1,kau2,kau3}, and written explicitly in the form we will use now by Krawitz \cite{Kr}. In what follows, we write 
\[
W_{g\cap h}:= W|_{\Fix g \cap \Fix h}
\]
and $\mu_{g\cap h}$ for the dimension of the corresponding Milnor algebra.

The B--model multiplication is defined on elements of the form $\fjrw{1}{g}$ and then extended to the entire state space multilinearly. On these particular basis elements, the definition of the product is 
\[
\fjrw{1}{g}\star_{\sB} \fjrw{1}{h} = \gamma_{g,h} \fjrw{1}{gh}, \text{ where}
\]
\[
\gamma_{g,h} \frac{\Hess (W_{g\cap h})}{\mu_{g\cap h}} = \left\{ \begin{array}{l l }
\frac{\Hess (W_{gh})}{\mu_{gh}} & \text{ if } I_g \cup I_h \cup I_{gh} = \{1, \ldots, N\}\\
0 & \text{ otherwise} \\ \end{array}\right.
\]

Notice also in this definition, that the product of two even-graded basis elements is also even-graded, so this defines a product on $\sB^0$ (see \cite{BTW}). 

\begin{rem}
	The definitions of \cite{BTW} and \cite{HLL} are similar to this definition, but have some other constants. 
\end{rem}

\subsection{A-- and B--model equivalence}

In what follows we will restrict our focus to those LG models $(W,\s{G})$ satisfying the following property:
\begin{property}\label{property}
	Let $W$ be a non-degenerate, invertible singularity, and let $G$ be an admissible group of symmetries of $W$. We say that \emph{the pair $(W,G)$ has \autoref{property}} if, 
	\begin{enumerate}
		\item $W$ can be decomposed as
		$   W = \sum_{i=1}^{M} W_i ,$
		where the $W_i$ are themselves invertible polynomials having no variables in common with any other $W_j$.  
		\item For any element $g$ of $G$ whose associated sector  $\sA_g \subseteq \sA_{W,G}$ is nonempty, and for each $i\in \{1,\dots ,M\}$ the action of $g$ fixes either all of the variables in $W_i$ or none of them. 
		\item For any element $g'$ of $G^T$ whose associated sector of $\sB_{g'} \subseteq \sB_{W^T,G^T}$ is non-empty, and for each $i\in \{1,\dots ,M\}$ the action of $g'$ fixes either all of the variables in $W_i^T$ or none of them.
	\end{enumerate}
\end{property}

In \cite{FJJS}, it was shown that for LG models $(W,G)$ satisfying \autoref{property}, we have an isomorphism of Frobenius algebras
\[
\sA_{W,G}\cong \sB_{W^T,G^T}, 
\]
given by a rescaling of the Krawitz map. 
The conditions are satisfied for a majority of cases---in particular for all polynomials with no chains (see \cite[Remark 1.1.1]{FJJS}).  
The result assumes that the A--model and B--model Frobenius algebras are commutative. However, one can see that the theorem still holds for the even-graded subspaces. 

We obtain the following result:
\begin{thm}\label{thm2}
If $W$ and $\s G$ satisfy \autoref{property}
then $\sA^0 \cong \sB^0$ as Frobenius algebras.
\end{thm}

To prove Theorem~\ref{thm2}, we will use the following lemma. 

\begin{lem}\label{lem:gamma_relations}
	\[
	\gamma_{g,h} = \left\{ 
	\begin{array}{ll}
	4 \gamma_{g, \sigma h} & \text{ if } g = (1/2, \alpha, 1/2, \beta),\ h = (1/2, \delta, 1/2, \nu)\\
	1/4 \gamma_{g, \sigma h} & \text{ if } g = (1/2, \alpha, 1/2, \beta),\ h = (0, \delta, 0, \nu)\\
	\gamma_{g, \sigma h} & \text{ otherwise} \end{array}\right.
	\]
\end{lem}
The straightforward proof is left to the reader.

\subsubsection{Proof of Theorem~\ref{thm2}}We will now exploit the isomorphism of Theorem~\ref{thm1} to provide an isomorphism of Frobenius algebras. 

We have the following diagram of Landau--Ginzburg models. The lower vertical (non--dashed) arrows come from Krawitz's mirror map. As mentioned above, under the conditions stated in Theorem~\ref{thm2}, there exist rescalings of these maps which yield the desired isomorphisms of Frobenius algebras. The horizontal arrows (and vertical double line) are the vector space isomorphisms of the previous section.   

\begin{center}\begin{tabular}{ccc}
		$\sB_{\tw{(W^T)},\tw{(G^T)}}$&\begin{tikzpicture} \draw[<-] (0,0) -- (1,0); \node at (.5,.25) {$tw_B$}; \end{tikzpicture} &$\sB_{W^T, \s{G^T}}$ \\
		\begin{tikzpicture} \draw[double] (0,0) -- (0,1); \node at (0,1.1) {}; \end{tikzpicture} & & \begin{tikzpicture} \draw[<->, dashed] (0,0) -- (0,1); \node at (0,1.1) {}; \end{tikzpicture} \\
		$\sB_{(\tw{W})^T,(\tw{G})^T}$ & &$\sB_{W^T,(\s{G})^T}$\\
		\begin{tikzpicture} \draw[<-] (0,0) -- (0,1); \node at (.7,1.1) {}; \node at (-.5,.5) {$\kappa_2$}; \end{tikzpicture} & & \begin{tikzpicture} \draw[<-] (0,0) -- (0,1); \node at (-.7,1.1) {}; \node at (.5,.5) {$\kappa_1$};\end{tikzpicture} \\
		$\sA_{\tw{W},\tw{G}}$&\begin{tikzpicture} \draw[<-] (0,0) -- (1,0); \node at (.5,.25) {$tw_A$};\end{tikzpicture} &$\sA_{W, \s{G}}$ \\
\end{tabular}\end{center}

Using this description as a guide, we will define a map $\varphi:\sB_{W^T,(\s{G})^T} \to \sB_{W^T,\s{G^T}}$, which is a rescaling of the composition $tw_B^{-1}\circ \kappa_2^{-1} \circ tw_A \circ \kappa_1$ of the solid arrows in the diagram, each of which is a degree--preserving isomorphism of graded vector spaces.  

Let $\sB^2$ denote the even-graded subspace of $\sB_{W^T, (\s{G})^T}$ (the BHK mirror). We then show that the restriction
\[
\bar\varphi:\sB^2\to\sB^0
\] 
preserves products and pairings. The isomorphism of Theorem~\ref{thm2} is the composition 
\[
\bar\varphi\circ\kappa_1^{-1}:\sA^0\to \sB^0
\]  

%

\begin{rem}\label{rem:multmir}
It is also worth noting that in the diagram, there are two \emph{different} B--models. This is related to the so--called multiple mirror phenomenon as mentioned in the Introduction. The B--model in the middle of the right-hand side is BHK mirror for $(W,\s{G})$. However, as we have seen, $\s{G^T}\neq (\s{G})^T$ and so the corresponding geometry does not fit the Borcea--Voisin construction. Because of the LG/CY correspondence, we expect a B--model whose corresponding geometry is Borcea--Voisin type, which is what the upper right--hand B--model is. Also, since we have two B--models, both mirror to the same A--model, it should be expected that the state spaces for the two B--models are isomorphic. 
\end{rem}

For what follows, 
we introduce the notation $\sB_{BHK} = \sB_{W^T, (\s{G})^T}$ (the BHK mirror) and $\sB_{BV} = \sB_{W^T, \s{G^T}}$ (the BV mirror). 

It is straightforward to verify that for any $\fjrw{m}{g} \in \sB_{BHK} $, where $g = (\epsilon/2, \alpha, \epsilon/2, \beta)$, $g' = (\alpha, \beta)$ and $m'$ is the same as $m$ except for the absence of the $dx_0 \wedge dy_0$ term (which may or may not be present in $m$). 
\[ \kappa_2^{-1} \circ tw_A \circ \kappa_1 \left( \fjrw{m}{g}\right) = k_{m,g} \fjrw{m'}{g'},\]
for some constant $k_{m,g}$.  Thus, composing finally with  $tw_B^{-1}$, gives the following map: 
\begin{equation}\label{algmap1}
\begin{array}{c}
tw_B^{-1}\circ \kappa_2^{-1} \circ tw_A \circ \kappa_1(\fjrw{m}{g})=\hspace*{6cm}\\ [2mm]
\hspace*{3cm}\left\{ \begin{array}{ll}
k_{{m},{g}}\fjrw{m'}{(1/2, \alpha, 1/2, \beta)} &  \text{ if }\jw_{x,B}\text{ acts with weight 0 on }m',\\
k_{{m},{g}}\fjrw{m'dx_0 \wedge dy_0}{(0, \alpha, 0, \beta)} & \text{ otherwise}, \end{array}\right.
\end{array}
\end{equation}
where $\jw_{x,B}$ is the exponential grading operator associated with $f_1^T$. 

Note that for $\fjrw{m}{g}\in \sB_{BHK}$,  $g$ is always an element of $\s{G^T}$, since $(\s{G})^T\subset \s{G^T}$. But $m$ may not be fixed by the action of all elements in the larger group.  We shall use $\tilde{m}$ to denote the monomial plus volume form which  differs from $m$ by the $dx_0\wedge dy_0$ volume form (if $m$ contains the term $dx_0\wedge dy_0$, then $\tilde{m}$ does not, and vice-versa). 
Lemma~\ref{l:allornothing} shows that exactly one of $\fjrw{m}{g}$ or $\fjrw{\tilde m}{\sigma g}$ is in $\sB_{BV}$.  The following lemma further clarifies the map in Equation~\eqref{algmap1}.

\begin{lem}
	Let $\fjrw{m}{g} \in \sB_{BHK}$, then 
	\[
	tw_B^{-1}\circ \kappa_2^{-1} \circ tw_A \circ \kappa_1(\fjrw{m}{g})= \left\{ \begin{array}{ll}
	k_{m,g}\fjrw{m}{g} &  \text{ if }\jw_{W_1^T}\text{ fixes }m,\\
	k_{m,g}\fjrw{\tilde{m}}{\sigma{g}} & \text{ otherwise}. \end{array}\right.
	\]
	Further, $\fjrw{m}{g} \in \sB_{BV}$ exactly when $\jw_{W_1^T}$ fixes $m$. 
\end{lem} 

\begin{proof}
Suppose that $\jw_{W_1^T}$ fixes $m$.  If, also, $\jw_{x,B}$ fixes $m$, then we know that $m = m_x m_y$ and $g = (1/2, \alpha, 1/2, \beta)$ for some $\alpha$ and $\beta$.  In this case $tw_B(\fjrw{m'}{g'}) = \fjrw{m}{g}$, as desired. If, instead $\jw_{x, B}$ acts with weight 1/2 on $m_x$, then it must be that $m = m_x m_y dx_0 \wedge dy_0$ and $g = (0, \alpha, 0, \beta)$.  Again, this yields $tw_B(\fjrw{m'}{g'}) = \fjrw{m}{g}$.
	
If $\jw_{W_1^T}$ doesn't fix $m$, a similar analysis of the cases where $\jw_{x,B}$ fixes/doesn't fix $m$ shows that $tw_B(\fjrw{m'}{g'}) = \fjrw{\tilde m}{\sigma g}$ in both cases. 
\end{proof}

If $\varphi$ is any rescaling of the map above then it will be a degree preserving bijection between B--models.  It remains to define a rescaling and demonstrate that it will preserve products and pairings, which we do as follows:

\[
\varphi(\fjrw{m}{g}) = \left\{ \begin{array}{ll} \fjrw{m}{g} & \text{if } \fjrw{m}{g} \in \sB_{W^T, \s{G^T}} \\ 
\tfrac 12 \fjrw{\tilde m}{\sigma g} & \text{if } \fjrw{m}{g} \notin \sB_{W^T, \s{G^T}}, g = (0, \alpha, 0, \beta) \\ 
2\fjrw{\tilde m}{\sigma g} & \text{if } \fjrw{m}{g} \notin \sB_{W^T, \s{G^T}}, g = (1/2, \alpha, 1/2, \beta). \\ 
\end{array}\right.
\]

Notice that if $\fjrw{m}{g}$ and $\fjrw{n}{h}$ are elements in $\sB_{BHK}$ with $I_g \cup I_h \cup I_{gh} = \{x_0, \ldots, x_m, y_0 \ldots y_n\}$ and both are elements of $\sB_{BV}$, then $\fjrw{mn}{gh}$  must also be an element of $\sB_{BV}$. Also, if neither is an element of $\sB_{BV}$, then $\fjrw{mn}{gh}$ is an element of $\sB_{BV}$ (Since $\fjrw{\tilde m}{\sigma g}$ and $\fjrw{\tilde n}{\sigma h}$ are in $\sB_{BV}$ and thus their product must be).  Only if one  of  $\fjrw{m}{g}$ and $\fjrw{n}{h}$ is in $\sB_{BV}$ and the other is not, do we find that $\fjrw{mn}{gh} \notin \sB_{BV}$. 

Now we restrict the map $\varphi$ to $\bar\varphi:\sB^2\to \sB^0$ and notice it is an isomorphism of bi-graded vector spaces, since $\varphi$ preserves bi-degree. We check that $\bar\varphi$ is an isomorphism of Frobenius algebras in two steps: first we check that it preserves products, and then we check that it preserves the pairing. 

\begin{lem} The map $\bar\varphi$ preserves products.
\end{lem}

\begin{proof}
We will denote the product on $\sB^2$ by $\star_2$ and the product on $\sB^0$ by $\star_0$. 
	Based on the definition of $\bar\varphi$, there are several cases to check. 
	
	\noindent \textbf{Case A:} Suppose that $\fjrw{m}{g}$ and $\fjrw{n}{h}$ are each elements of both $\sB^2$
	and $\sB^0$. Based on the previous remarks, $\fjrw{mn}{gh}$ is in both as well, and so 
	\[
	\begin{array}{l}
	\bar\varphi(\fjrw{m}{g} \star_2 \fjrw{n}{h}) = \bar\varphi( \gamma_{g,h} \fjrw{mn}{gh}) = \gamma_{g,h} \fjrw{mn}{gh} \\
	=
	\fjrw{m}{g} \star_0 \fjrw{n}{h}=\bar\varphi(\fjrw{m}{g}) \star_0 \bar\varphi(\fjrw{n}{h})
	\end{array}
	\]
	
	\noindent \textbf{Case B:} Next, suppose that $\fjrw{m}{g}$ and $\fjrw{n}{h}$ are each elements of $\sB^2$, but neither are elements of $\sB^0$.   
	Notice that in this case 
	\[
	\bar\varphi(\fjrw{m}{g}) \star_0 \bar\varphi(\fjrw{ n}{ h})=k_g\fjrw{\tilde m}{\sigma g} \star_0 k_h\fjrw{\tilde n}{\sigma h} = k_gk_h\gamma_{\sigma g, \sigma h} \fjrw{mn}{ gh},
	\]
	since $\sigma g \sigma h = gh$.  
	Thus, to verify that  
	$
	\bar\varphi(\fjrw{m}{ g}) \star_0  \bar\varphi( \fjrw{n}{ h}) =\bar\varphi(\fjrw{m}{ g} \star_2   \fjrw{n}{ h})$, 
	we only need to check that 
	\[
	k_gk_h\gamma_{\sigma g, \sigma h} \fjrw{mn}{gh}= \gamma_{g, h} \fjrw{mn}{ gh}.
	\]

	There are three relevant cases. 
	\begin{enumerate}
		\item $g$ and $h$ are both of the form $(0, \alpha, 0, \beta)$
		\item $g$ and $h$ are both of the form $(1/2, \alpha, 1/2, \beta)$
		\item $g$ is of one of the above forms and $h$ is of the other.
	\end{enumerate}
	
	It is straightforward to verify that in each case, 
	$I_g \cup I_h \cup I_{gh} = \{x_0, \ldots, x_m,y_0,\ldots,y_n\}$
	if and only if 
	$I_{\sigma g} \cup I_{\sigma h} \cup I_{gh} = \{x_0, \ldots, x_m,y_0,\ldots,y_n\}$.

	We verify the details of the preservation of the product in only the first case; the others follow by similar arguments. 
	\begin{enumerate}
		\item  $g$ and $h$ are both of the form $(0, \alpha, 0, \beta)$. \\
		Here, $k_g = k_h = 1/2$. Suppose that $I_g \cup I_h \cup I_{gh} = \{x_0, \ldots, x_m,y_0,\ldots,y_n\}$. 
		Lemma \autoref{lem:gamma_relations} gives 
		\[
		\gamma_{g,h}  = \gamma_{g,\sigma h} 
		=\tfrac 14  \gamma_{\sigma g,\sigma h}
		\]
		
		So, we have, $k_gk_h\gamma_{\sigma g, \sigma h} = 1/4 \cdot \gamma_{\sigma g,\sigma h} = \gamma_{g,h}$, as desired. 

	\end{enumerate}

	\noindent \textbf{Case C:} Finally, we must consider the case when $\fjrw{m}{g}, \fjrw{n}{h}$ are both elements of $\sB^2$, but only one of them is an element of $\sB^0$.  Without loss of generality, we say $\fjrw{m}{g} \in \sB^0$.  Notice that in this case
	\[
	\bar\varphi(\fjrw{m}{g} \star_2 \fjrw{n}{h}) = \bar\varphi(\gamma_{g,h} \fjrw{mn}{gh}) = \gamma_{g,h}k_{gh} \fjrw{\widetilde{mn}} {\sigma gh},
	\]
	and, 
	\[
	\bar\varphi(\fjrw{m}{g}) \star_0 \bar\varphi(\fjrw{n}{h}) = \fjrw{m}{g} \star_0 k_h \fjrw{\tilde n}{\sigma h} = \gamma_{g,\sigma h}k_{h} \fjrw{\widetilde{mn}} {\sigma gh}.
	\]
	We must prove $\gamma_{g,h}k_{gh} \fjrw{mn} {\sigma gh} =\gamma_{g,\sigma h}k_{h} \fjrw{mn} {\sigma gh}$ in  the following  four cases: 
	\begin{enumerate}
		\item $g$ and $h$ are both of the form $(0, \alpha, 0, \beta)$ ($k_h =k_{gh}= 1/2$)
		\item $g$ and $h$ are both of the form $(1/2, \alpha, 1/2, \beta)$ ($k_h = 2, k_{gh}= 1/2$).
		\item $g$ is of the form $(0, \alpha, 0, \beta)$, and $h$ is of the form $(1/2, \alpha, 1/2, \beta)$ ($k_h = k_{gh}= 2$)
		\item $g$ is of the form $(1/2, \alpha, 1/2, \beta)$, and $h$ is of the form $(0, \alpha, 0, \beta)$ ($k_h= 1/2,  k_{gh}= 2$).
	\end{enumerate}
	
	Again, we verify the preservation of the product in the first case only; the others follow similarly. 
	\begin{enumerate}
		\item $g$ and $h$ are both of the form $(0, \alpha, 0, \beta)$ ($k_h =k_{gh}= 1/2$).  From Lemma \ref{lem:gamma_relations}, 
		$
		\gamma_{g,h}  
		=  \gamma_{g,\sigma h}.
		$
		So, 
		$k_{gh}\gamma_{g,h} = \tfrac 12 \gamma_{g,h} = \tfrac 12 \gamma_{g,\sigma h} = k_{h}\gamma_{g,\sigma h}$, as desired. 

	\end{enumerate}
	
\end{proof}

\begin{rem}
	Recall that for an element $\fjrw{m}{g}$ from an LG state space, $m$ represents a monomial together with a volume form.  In the following proof, we are sometimes interested in only the monomial portion of $m$ (which by abuse of notation, we will call $m$).  Further, we point out that the monomial portion of both $m$ and $\tilde{m}$ are the same. 
\end{rem}

\begin{lem}The map $\bar\varphi$ preserves the pairing. 
\end{lem}

\begin{proof}
	Recall that if $\langle \fjrw{m}{g}, \fjrw{n}{g^{-1}}\rangle_{\sB^2} \neq 0$, then we have 
	\[
	mn = \frac{\langle \fjrw{m}{g}, \fjrw{n}{g^{-1}}\rangle_{\sB^2}}{\mu_g}\Hess(W_g) + l.o.t
	\]
	If both $\fjrw{m}{g}$ and  $\fjrw{n}{g^{-1}}$ are in $\sB^0$, then the pairings are computed in exactly the same way in both $\sB^2$ and $\sB^0$ 
	
	If neither $\fjrw{m}{g}$ nor $\fjrw{n}{g^{-1}}$ are in $\sB^0$, then we have two cases to consider. 
	\begin{enumerate}
		\item $g$ is of the form $(0, \alpha, 0, \beta)$. \\
		\[
		m \cdot n = \frac{\langle \fjrw{m}{g}, \fjrw{n}{g^{-1}} \rangle_{\sB^2}}{\mu_g} \Hess(W_g)  = \frac{\langle \fjrw{m}{g}, \fjrw{n}{g^{-1}} \rangle_{\sB^2}}{\mu_{\sigma g}} 4 \Hess(W_{\sigma g}) 
		\]
		Thus, for $\bar\varphi(\fjrw{m}{g}) = \tfrac 12 \fjrw{\tilde m}{\sigma g}$, and  $\bar\varphi(\fjrw{n}{g^{-1}}) = \tfrac 12 \fjrw{\tilde n}{\sigma g^{-1}}$
		\[
		\tfrac 12 m \cdot \tfrac 12 n   = \frac{\langle \fjrw{m}{g}, \fjrw{n}{g^{-1}} \rangle_{\sB^2}}{\mu_{\sigma g}}  \Hess(W_{\sigma g}),
		\]
		and $ \langle \fjrw{m}{g}, \fjrw{n}{g^{-1}} \rangle_{\sB^2} = \langle \bar\varphi(\fjrw{m}{g}), \bar\varphi(\fjrw{n}{g^{-1}}) \rangle_{\sB^0}$.
		
		\item $g$ is of the form $(1/2, \alpha, 1/2, \beta)$. \\
		\[
		m \cdot n = \frac{\langle \fjrw{m}{g}, \fjrw{n}{g^{-1}} \rangle_{\sB^2}}{\mu_g} \Hess(W_g)  = \frac{\langle \fjrw{m}{g}, \fjrw{n}{g^{-1}} \rangle_{\sB^2}}{\mu_{\sigma g}} \tfrac 14 \Hess(W_{\sigma g}) 
		\]
		Thus, for $\bar\varphi(\fjrw{m}{g}) = 2\fjrw{\tilde m}{\sigma g}$, and  $\bar\varphi(\fjrw{n}{g^{-1}}) = 2\fjrw{\tilde n}{\sigma g^{-1}}$
		\[
		2m \cdot 2n   = \frac{\langle \fjrw{m}{g}, \fjrw{n}{g^{-1}} \rangle_{\sB^2}}{\mu_{\sigma g}}  \Hess(W_{\sigma g}). 
		\]
		So, $ \langle \fjrw{m}{g}, \fjrw{n}{g^{-1}} \rangle_{\sB^2} = \langle \bar\varphi(\fjrw{m}{g}), \bar\varphi(\fjrw{n}{g^{-1}}) \rangle_{\sB^0}$.
	\end{enumerate}
	
	Note that if $\fjrw{m}{g}$ is in $\sB^0$ but  $\fjrw{n}{g^{-1}}$ is not, then $\langle\fjrw{m}{g},\fjrw{n}{g^{-1}}\rangle=0$, since the $\jw_x$ acts with weight zero on the Hessian. 
	
\end{proof}

We have verified that $\bar\varphi$ preserves both the product and the pairing, and therefore, the composition, $\bar\varphi \circ \kappa_1^{-1}$ is an isomorphism of Frobenius algebras when \autoref{property} is satisfied.

\bibliographystyle{plain}
\bibliography{references}
\end{document}